\documentclass[12pt]{amsart}
\usepackage{amssymb, amsmath}
\usepackage{amsmath,amssymb,amsthm,mathrsfs,graphicx,url}
\DeclareMathAlphabet{\mathpzc}{OT1}{pzc}{m}{it}
\usepackage{graphics}
\usepackage{epsfig}
\usepackage{amsxtra}
\usepackage{color}
\usepackage{amsxtra}

\usepackage[letterpaper,hmargin=1in,vmargin=1.2in]{geometry}

\numberwithin{equation}{section}

\newtheorem{thm}{Theorem}[section]

\newtheorem{lemma}[thm]{Lemma}

\newtheorem{prop}[thm]{Proposition}
\newtheorem{cor}[thm]{Corollary}
\newtheorem{hyp}[thm]{Hypothesis}

\theoremstyle{definition}

\newcommand{\defeq}{\stackrel{\rm{def}}{=}}

\def \msca {\operatorname{m}_{sc}}
\renewcommand{\Re}{\operatorname{\rm Re}\nolimits}
\renewcommand{\Im}{\operatorname{\rm Im}\nolimits}

\def \rank {\operatorname{rank}}

\def \supp {\operatorname{supp}}

\def \sgn {\operatorname{sgn}}

\def \supp {\operatorname{supp}}
\def \tr {\operatorname{tr}}
\def \Vol {\operatorname{Vol}}

\def \vol {{\rm vol}}

\def \fn {\mathfrak n}

\def \restrict {\upharpoonright}

\def \Real {{\mathbb R}}

\def \Sphere {\mathbb{S}}
\def \Complex {\mathbb{C}}
\def \Natural {{\mathbb N}}
\def \Sphere {{\mathbb S}}

\def \U+ {U_+}
\def \Integers {{\mathbb Z}}

\def \pr {\mathpzc{p}}
\title
 [Resonant rigidity in even dimensions]
{
Resonant rigidity for  Schr\"odinger operators in even dimensions
}
   \author { T.J. Christiansen}
\address{Department of Mathematics,
University of Missouri,
Columbia, Missouri 65211, USA} 
\email{christiansent@missouri.edu}

\begin{document}

\begin{abstract}
This paper studies the resonances of 
Schr\"odinger operators with bounded, compactly supported,
real-valued potentials on $\Real^d$, where the dimension $d$ is even.  
If the potential $V$ is non-trivial and $d\not =4$, then the
meromorphic continuation of the resolvent of the  Schr\"odinger
operator has infinitely many poles, with
a quantitative lower bound on their density. 
A somewhat weaker statement holds if $d=4$. We 
prove several inverse-type results.  If the 
meromorphic continuations of the resolvents of two Schr\"odinger
operators $-\Delta +V_1$ and $-\Delta +V_2$ have the same poles,
$V_1,\;V_2\in L^\infty_c(\Real^d;\Real)$,
$k\in \Natural$ and if $V_1\in H^k(\Real^d;\Real)$, then $V_2\in H^k$ as well.  
Moreover, we prove that certain sets of isoresonant potentials are compact.
We also
show that
the poles of the resolvent for a smooth potential
 determine the heat coefficients
and that the (resolvent) resonance sets of two  potentials
in $L^\infty_c(\Real^d;\Real)$ cannot
differ by a nonzero finite number of elements away from $0$.
\end{abstract}

\maketitle
\section{Introduction}


This paper  proves some results about 
resonances of the Schr\"odinger operator $-\Delta +V$ on $\Real^d$ 
when $d$ is even and the potential $V\in L^\infty_c(\Real^d;\Real)$.  For example, we show that if 
$V\in L^\infty_c(\Real^d;\Real)$
is nontrivial and $d\not =4$ then $-\Delta +V$ has infinitely many 
(resolvent) resonances.  If $d=4$, then either $0$ is a resonance of 
$-\Delta+V$, or there are infinitely many resonances.  
Hence, one could say that this demonstrates the 
resonant rigidity of the $0$ potential among all potentials in $L^\infty_c(\Real^d;\Real)$.
If $V_1,V_2\in L^\infty_c(\Real^d;\Real)$ have the same 
resolvent resonance
set, including multiplicities,
 and if $V_1\in H^k(\Real^d) $ for some $k\in \Natural$, then $V_2\in H^k(\Real^d)$ as well.  These
results are inspired by analogous results of Smith and Zworski \cite{sm-zw} 
in  odd dimension $d\geq 3$.  In addition, we show that if
$ V_1,\; V_2\in C_c^{\infty}(\Real^d;\Real)$ have 
the same resonances, including multiplicities, then they 
have the same heat coefficients.  
The compactness of the set of potentials in $L^\infty_c(\Real^2;\Real)$
 with support in a fixed compact set
and  having the same poles
as a fixed potential $V_0\in C_c^{\infty}(\Real^2;\Real)$ then 
follows rather directly
by results of \cite{bruning, donnelly}.
 There is 
 a weaker result in higher dimensions.  See \cite{hi-wo} for analogous 
results in dimension $d=1,3$.  
As a whole, these results can be interpreted as saying something about the 
rigidity of the set of potentials $V\in L^\infty_c(\Real^d;\Real)$ having
the same resonances.


For Schr\"odinger operators on $\Real^d$, the manifold on which the 
resonances lie is determined by the parity of the dimension $d$.
Set $R_V(\lambda)=(-\Delta+V-\lambda^2)^{-1}$ when $0<\arg \lambda<
\pi$.  Then if $d$ is odd $R_V$ has a meromorphic continuation, as an operator
from $L^2_c(\Real^d)$ to $L^2_{\operatorname{loc}}(\Real^d)$, to the complex plane.
If $d$ is even, the continuation is to $\Lambda$, the logarithmic
cover of $\Complex \setminus \{0\}$.  The poles of this continuation to 
$\Lambda$
are the nonzero (resolvent) resonances. 
 Here we explicitly include as resonances
poles 
corresponding to eigenvalues and lying in the physical
space $\{ \lambda: 0<\arg \lambda <\pi\}$,
 although conventions differ on this.

The following theorem provides a 
quantitative lower bound on the number of resonances for a Schr\"odinger
operator in even dimensions.  We describe a point $\lambda\in \Lambda$
by specifying its norm $|\lambda|$ and argument $\arg \lambda$, where
we do not identify points whose arguments differ by a nonzero integral
multiple of $2\pi$.   For Hypothesis \ref{hy:H1}, which is a hypothesis about the nature of the singularity of the resolvent at the origin (if it is unbounded
 there), 
see Section \ref{ss:H1}.  We remark
that Hypothesis \ref{hy:H1} holds generically.
  The definition of the multiplicity of
a nonzero pole of the resolvent is given in (\ref{eq:multresolvent}).
\begin{thm}\label{thm:infinitelymany} Let $d$ be even.  
Let $V\in L^{\infty}_c(\Real^d;\Real)$ and suppose $V\not \equiv 0$.
If $d=4$, suppose in addition that Hypothesis \ref{hy:H1} holds.
Then  $-\Delta +V$ has infinitely many (resolvent) resonances.  In fact, 
set $N(r)$ to be the number of poles of the resolvent on $\Lambda$,
counted with multiplicity, that have $1/r<|\lambda|<r$ and 
$|\arg \lambda|<\log r$.  Then, for any $\epsilon >0$,
$$\lim \sup_{r\rightarrow \infty}
\frac{N(r)}{(\log r)^{1-\epsilon}}=\infty.$$
\end{thm}
S\'a Barreto \cite{SBeven} ($d\geq 4$) and Chen \cite{chen} ($d=2$)
proved a  related,  stronger,
bound for $V\in C_c^\infty(\Real^d;\Real)$, $V\not \equiv 0$.  
They showed $\lim \sup _{r\rightarrow \infty}
\frac{N(r)}{(\log r)(\log \log r)^{-p}}=\infty$ for any $p>1$,
and did not require Hypothesis \ref{hy:H1} when $d=4$.  Each of these lower bounds
is much smaller than the upper bounds  known to hold in even dimensions
 \cite{intissar, vodeveven, vodev2},
and the lower bounds  which are known to hold generically, see 
\cite{ch-hi2}.

Interestingly, in {\em odd} dimensions $d> 3$, the result 
of \cite{sm-zw} analogous
to Theorem \ref{thm:infinitelymany}
 is that any non-trivial, real-valued potential
$V$ must have at least one resonance, although in
odd dimensions it is known that any non-trivial, smooth, real-valued
potential must have infinitely many, and 
there is a quantitative lower bound (e.g. \cite{SaB01} and references
therein).  
In dimension $3$, any nontrivial
$V\in L^{\infty}_c(\Real^3;\Real)$ must have infinitely many (\cite{sm-zw}), and in 
dimension $d=1$, asymptotics of the 
resonance-counting function are known \cite{froese, simon, zworski1d}.
  Both 
here and in \cite{sm-zw}, it is important that we require 
$V$ to be real-valued, since in dimension at least $3$ there are 
examples of complex-valued potentials with no resonances,
and in dimension $d=2$ no
resonances away from $0$ \cite{autin,chex,iso}.  See \cite{sm-zw} and references therein for further results in odd dimensions.

 It is important to emphasize that 
in the assumptions of Theorems \ref{thm:heatcoefficients} and 
\ref{thm:inverse} and elsewehere 
we use the notions of multiplicity of a resonance
defined in this paper.  The multiplicity of points of $\Lambda$ as
a resonance is rather standard and is recalled in Section 
\ref{s:multiplicities}.
However, the notion of the multiplicity of $0$ as a resonance is a 
rather subtle point.  The one we use here might more 
properly be called a normalized or weighted multiplicity, and 
can be found in Section \ref{ss:H1}.
    For 
other purposes a different notion of multiplicity of $0$ as a resonance
than that of this paper
may be preferable.  

The preliminary steps in the proof of Theorem \ref{thm:infinitelymany} 
give rather directly
 some results about the heat coefficients for smooth potentials
in even dimensions.  Recall that here and elsewhere we include poles of
the resolvent corresponding to eigenvalues of $-\Delta +V$  among
the resonances.
\begin{thm}\label{thm:heatcoefficients}  Let $d$ be even, and 
let $V_1,\; V_2\in C_c^\infty(\Real^d;\Real)$.  Suppose
$V_1$ and $V_2$ have the same resolvent
resonances, including multiplicities.  
Then $V_1$ and $V_2$ have the same heat coefficients.
\end{thm}

A similar result holds in odd dimension $d$.  Two 
potentials in $C_c^\infty(\Real^3;\Real)$ with the same resonance set which
does not include $0$ have
the same heat coefficients, except, possibly, for the first--that is, the
integral of $V$.  If $d\geq 5$, the first two heat coefficients may differ
if $0$ is an eigenvalue of both Schr\"odinger operators.
This follows from  bounds on the determinant of the 
scattering matrix and the number of 
resonances, Hadamard's factorization theorem, a trace formula,
 and the behavior of the determinant of the
scattering matrix near $0$.  For a different proof  in dimensions
$d=1$ and $d=3$, see \cite{hi-wo}.

For $R>0$, set $B(0,R)=\{x\in \Real^d:|x|<R\}$.
Fix $R_0>0$, and let $V_0\in C_c^\infty(\Real^d;\Real)$ satisfy $
\supp V_0 \subset \overline{B}(0,R_0)$.
  Set 
\begin{multline}\label{eq:isodef}
Iso(V_0,R_0)=\{ V\in C^\infty_c(\Real^d;\Real): \; \supp V \subset \overline{B}(0,R_0)\;
\text{ and
$-\Delta+V$ and $-\Delta+V_0$}\\ \text{have the same resolvent resonances, including multiplicities}\}
\end{multline}
and, for $c_0>0$, $s\geq 0$, 
$$Iso(V_0,R_0, s,c_0)= \{ V\in Iso(V_0,R_0):\; \| V\|_{H^s}\leq c_0\}.$$
\begin{thm}\label{thm:inverse} 
Let $V_0\in C_c^\infty(\Real^d;\Real)$, with 
$\supp V_0\subset B(0,R_0)$.  If $d=2$, then $Iso(V_0,R_0)$ is compact in the 
topology of $C^\infty(B(0,R_0))$.  If $d\geq 4 $ is even and $s>d/2-2$, then
$Iso(V_0,R_0,s,c_0)$ is compact in $C^\infty(B(0,R_0))$. 
\end{thm}

Hislop and Wolf have proved analogs of Theorems \ref{thm:heatcoefficients}
and \ref{thm:inverse} in dimensions $d=1$ and $d=3$, \cite{hi-wo}.
In dimension $1$, some stronger results are
due to Zworski \cite{zwremark}
and Korotyaev \cite{khalfline, kstability,kline}.
 Our proof of Theorem
\ref{thm:inverse} uses results of Br\"uning \cite{bruning}
and Donnelly \cite{donnelly} for isospectral Schr\"odinger operators
on compact Riemannian manifolds, together with 
Theorem \ref{thm:heatcoefficients}.  We remark that again it is 
necessary to assume the potentials are real-valued, as examples
of \cite{autin, iso} give large families of isoresonant complex-valued 
potentials which are not even bounded in $L^\infty$.


Our proof of Theorem \ref{thm:infinitelymany}
 uses an adaptation of techniques from \cite{sm-zw} and \cite{SBeven}.
The central novel technical results are Theorems
\ref{thm:samepoles} and  \ref{thm:derivasmpt}.
Theorem \ref{thm:samepoles}  gives a  relationship between the 
determinants of two scattering matrices 
if the difference of the sets of their poles is not too big.  
When combined  with techniques of \cite{sm-zw} or \cite{SBeven},
we shall see that Theorem \ref{thm:samepoles} has a number of corollaries,
among them Theorems \ref{thm:infinitelymany}, \ref{thm:heatcoefficients},
 and  \ref{thm:Hm}.

In the statement of Theorem \ref{thm:samepoles}, for $z\in \Complex$,
$S_j(e^z)$ means we evaluate $S_j$ at the point
in $\Lambda$ having argument $\Im z$ and norm $e^{\Re z}$.
\begin{thm}\label{thm:samepoles}
Let $d$ be even, and  $V_j\in L^{\infty}_c(\Real^d;\Real)$ for $j=1,\;2.$  Set
$P_j=-\Delta+V_j$ and let $S_j(\lambda)$ be the associated scattering 
matrix, unitary for $\lambda>0$. Set
$$F(z)=\frac{\det S_1(e^z)}{\det S_2(e^z)}.$$
Let $\{z_l\}$ denote the distinct poles of $F(z)$, and $M(z_l)$ their 
multiplicities. Suppose that for some $\epsilon_0>0$,
\begin{equation} \label{eq:notmanypoles1}
\sum _{|z_l|<r:\; M(z_l)>0} 1=O(r^{1-\epsilon_0})
\end{equation}
and 
\begin{equation}\label{eq:notmanypoles2}
\sum \frac{M(z_l)}{|z_l|^m}<\infty \; \text{ for some $m\in (0,\infty)$.}
\end{equation}
If the dimension $d=4$, assume in addition that {\em either}
 Hypothesis \ref{hy:H1} holds for both $P_1$ and $P_2$ {\em or} 
both $V_1, \; V_2\in C^\infty_c(\Real^4;\Real)$.
Under these hypotheses $F(z)$ has no poles, and 
\begin{equation}\label{eq:detrelation}
\det S_1(\lambda) =
\det S_2(\lambda) \; \text{for all $\lambda \in \Lambda$}.
\end{equation}
\end{thm}  

The multiplicity $M(z_l)$ used in  Theorem \ref{thm:samepoles} is 
given by 
$$M(z_l)=\min \{ m\in \Natural \cup\{0\}: (z-z_l)^m F(z)\; \text{is analytic
near $z=z_l$}\}.$$

Since, for $x,y\in \Real$, 
\begin{equation}
S_j^*(e^{x-iy})= (S_j(e^{x+iy}))^{-1}
\end{equation}
the condition on the poles of $F$ is  symmetric in $V_1$ and $V_2$,
although it may not appear to be so on first inspection. Moreover,
since $e^0=1$, and the scattering matrices are 
unitary on the real axis, $z=0$ is not a pole of $F$ so there 
is no difficulty with (\ref{eq:notmanypoles2}) at $0$.


 An analog
of Theorem \ref{thm:samepoles}
in the case of $V_2\equiv 0$ and $V_1\in C_c^\infty(\Real^d;\Real)$ is part of the proof of \cite[Theorem 1.1]{SBeven}.
 The extension to
arbitrary $V_1,\;V_2 \in C_c^\infty(\Real^d;\Real)$ is fairly straightforward;
 it is the limited
regularity assumed of the $V_j$ which is the delicate
issue to address in the proof of this result, and for this we need
Theorem \ref{thm:derivasmpt}.  The paper \cite{SBeven} uses in a central way that for 
$V\in C_c^\infty(\Real^d;\Real)$, the logarithmic derivative
of the determinant of the scattering matrix has an asymptotic expansion as
 $\lambda \rightarrow
\infty$, $\lambda \in \Real$, with error $O(\lambda^{-N})$ for any $N$,
see e.g. \cite{CdV,gu}.   
This is not available for potentials
$V$ which are merely in $L^\infty_c(\Real^d;\Real)$.  Theorem
\ref{thm:derivasmpt} provides the results needed.


Theorem \ref{thm:samepoles} and its corollary Theorem
\ref{thm:Hm} are stated, roughly, as results 
about the number of  points of $\Lambda$ which are poles of the 
determinant of one scattering 
matrix $S_j$ but not of the determinant of the other scattering matrix.  
In Section \ref{s:multiplicities}
 we recall some results about the relation between the poles
of the determinant of the scattering matrix and the poles of the meromorphic
continuation of the cut-off resolvent.  In particular, we note here that
$\det S_j(\lambda)$ cannot have a pole at $\lambda_0$ unless the meromorphic
continuation of the cut-off resolvent, $\chi R_j(\lambda)\chi$ has a pole
at $\lambda_0$.

Theorem \ref{thm:samepoles} has several corollaries, including
Theorems \ref{thm:infinitelymany}, \ref{thm:heatcoefficients},
and \ref{thm:Hm}.  Some of these are inspired by
analogous results of Smith and Zworski \cite{sm-zw} in the case of $d\geq 3$ odd.
\begin{thm}\label{thm:Hm} Let $d$ be even. Suppose $V_1$, $V_2\in L^{\infty}_c(\Real^d;\Real)$
are as in the statement of Theorem \ref{thm:samepoles}, including the 
conditions on the poles of $F$. If $d=4$, 
assume that Hypothesis \ref{hy:H1} holds for $P_1$ and $P_2$.
 If for some 
$k\in \Natural$, $V_2\in H^k(\Real^d)$, then $V_1\in H^k(\Real^d)$ as well.
\end{thm}
Again, this result is not true if we omit the hypothesis that the potentials
are real-valued.  The paper \cite{sm-zw} proves a similar result
in odd dimensions: if $V_1,\; V_2\in L^\infty_c(\Real^d;\Real)$ have the 
same resonances, including multiplicities, and if $V_1\in H^k(\Real^d)$, 
then $V_2\in H^k(\Real^d)$.

We note that a consequence of our Theorem \ref{thm:Hm} is
that for $V_0\in C_c^\infty(\Real^d;\Real)$ supported in $\overline{B}(0,R_0)$,
we could replace the definition (\ref{eq:isodef}) by the equivalent
\begin{multline*}
Iso(V_0,R_0)=\{ V\in L^\infty_c(\Real^d;\Real): \; \supp V \subset \overline{B}(0,R_0)\;
\text{ and
$-\Delta+V$ and $-\Delta+V_0$}\\ \text{have the same resolvent resonances, including multiplicities}\}.
\end{multline*}

Another corollary of Theorem \ref{thm:samepoles} is the following.
\begin{cor}\label{cor:infinitediff}
 Let $d$ be even, and let $V_1, \; V_2\in L^{\infty}_c(\Real^d;\Real)$.  Then, if $d\not = 4$,
the resolvent resonance sets of $V_1$ and $V_2$ cannot differ by a nonzero finite number of nonzero elements.
If $d=4$, the same is true, provided that Hypothesis \ref{hy:H1} holds for both potentials, or both potentials are smooth.
\end{cor}
We emphasize here that we cannot at this point exclude the 
possibility that the resonance sets of $V_1 $ and $V_2$ are the same 
{\em except} that the point $0$ has different multiplicities as 
an element of the resonance sets for the two potentials.  We remark
that Korotyaev has studied the rigidity of the resonance set for 
a larger class of potentials in one dimension, i.e., $d=1$.  He showed
that within this larger class of potentials, 
finitely many resonances can be ``moved''
within certain restrictions.  For the full line case see
 \cite[Theorem 1.3]{kline} and for the
half-line  \cite{khalfline, kstability}.

\subsection{Notational conventions}
Throughout this paper we shall use the convention that
$C$  stands for a positive constant, the value of which may change from line to line
without comment. By the physical space in $\Lambda$  we mean the copy of the 
upper half plane $\{\Im \lambda>0\}$ in $\Lambda$ on which the 
resolvent $R_V(\lambda)$ is bounded from $L^2(\Real^d)$ to $L^2(\Real^d)$.
Our convention is that this corresponds to $\{ \lambda\in \Lambda: 0<\arg
\lambda<\pi\}$, and we identify this with the upper half plane when convenient.
When $(0,\infty)=\Real_+$ is considered as a subset of $\Lambda$, it is 
identified with the set of points with argument $0$.  Likewise, if $\lambda 
\in\Lambda$, by
$\lambda>0$ or by $\lambda \in (0,\infty)$ we mean that the point 
$\lambda \in \Lambda$ has argument $0$.  

The set of resonances of $-\Delta +V$ includes all poles of $R_V$ on $\Lambda$,
including those corresponding to eigenvalues, and should be repeated 
with multiplicity.  Moreover, $0$ may be a resonance of $-\Delta +V$,
with multiplicity as defined in Section 
\ref{ss:H1}.

The dimension $d$ is assumed to be
even in subsequent sections, {\em except}  in Section \ref{s:asymptotic},
where $d$ can be even or odd, but $d\geq 2$.

The set $L^\infty_c(\Real^d;\Real)=\{ f\in L^\infty(\Real^d;\Real):\; \text{$f$ has compact support}\}$,
and similarly for $L^2_c(\Real^d)$, $C^\infty_c(\Real^d)$.

\subsection{Organization of the paper}
We briefly outline the organization of this paper. 
In Section \ref{ss:H1} we discuss the behavior of the resolvent near $0$
and fix the notion of the multiplicity of $0$ as a resonance which we shall
use in this paper.  This section also includes Hypothesis \ref{hy:H1}.
Section \ref{s:multiplicities}  recalls some results about the 
relationship between the poles and zeros of the determinant of the 
scattering matrix and the poles of the 
(meromorphically continued) resolvent,
and recalls the definition of the multiplicity of a nonzero resonance.  

An important technical step in the proof of Theorem \ref{thm:samepoles}
is Theorem \ref{thm:derivasmpt}, which is proved in Section 
\ref{s:asymptotic}.  Theorem  \ref{thm:derivasmpt} is a result
on the high-energy behavior of the logarithmic derivative of the 
determinant of the scattering matrix of $-\Delta +V$ under the
assumption only that $V\in L^\infty_c(\Real^d;\Real)$--that is, without
any additional regularity assumed on $V$.  Theorem
\ref{thm:derivasmpt} is valid in both even and odd dimensions.  
See Section \ref{s:asymptotic} for the statement of the theorem 
and references to earlier results.  

In Section \ref{s:near0} we turn to the behavior of the determinant of the 
scattering matrix near $0$.  
In Section \ref{s:canonicalprod} we write the function $F$ from
Theorem \ref{thm:samepoles} using a canonical product, and
then prove Theorem \ref{thm:samepoles} in Section \ref{s:samepoles}. The proof uses results about both 
the high and low energy behavior of the logarithmic derivative of the scattering matrix.  
Section \ref{s:thmpfs} includes proofs of Theorems \ref{thm:Hm}, 
\ref{thm:infinitelymany}, and 
 \ref{thm:heatcoefficients} and Corollary \ref{cor:infinitediff}.

In Section \ref{s:independence}
we 
consider some questions related to linear independence of elements of the 
image of the singular part of the resolvent at different points on $\Lambda$.


\section{Resonances at $0$ and Hypothesis \ref{hy:H1}}
\label{ss:H1}

Let $V\in L^\infty_c(\Real^d;\Real)$ with $d$ even.
All but at most
one of the resonances of $-\Delta +V$ lie on $\Lambda$, the logarithmic
cover of $\Complex \setminus \{0\}$.  However, $0$ may be a resonance
as well.  Since this is a delicate question and is related to Hypothesis
\ref{hy:H1}, we address this here.  For further details in even
dimensions we refer the reader to \cite{b-g-d} for dimension $2$,  \cite{jensen4} for dimension $4$, and \cite{jensen>4} for dimension $d\geq 6$.  We 
recall the results which are most important to us here, clarifying the 
language we shall use to describe the possible scenarios.

If for each $\chi  \in C_c^\infty(\Real^d)$, $\lim_{\epsilon \downarrow 0}\chi (-\Delta+V+\epsilon^2)^{-1}\chi$ exists as a bounded operator on $L^2(\Real^d)$, then
$0$ is {\em not} a resonance of $-\Delta +V$.  If this limit does not 
exist for some
$\chi \in C_c^\infty(\Real^d)$, then our convention is that $0$ is a resonance of $-\Delta +V$,
even if this singularity is caused by $-\Delta +V$ having $0$ as an eigenvalue.

Let  $\mathcal{P}_0$ denote
projection onto the $L^2$ null space of $-\Delta +V$, with 
$\mathcal{P}_0=0$ if $0$ is not an eigenvalue of $-\Delta +V$.
  Then if $\mathcal{P}_0\not =0$ the leading 
singularity of $(-\Delta+V+\epsilon^2)^{-1}$ is given by $\mathcal{P}_0 \epsilon^{-2}$.  

If $d\geq 6$, then for any 
$\chi \in C_c^\infty(\Real^d)$
\begin{equation}\label{eq:0beh}
\lim_{\epsilon \downarrow 0} \chi(  (-\Delta +V+\epsilon^2)-\mathcal{P}_0/\epsilon^2) \chi
\end{equation}
exists, \cite{jensen>4}.  However, the behavior of the resolvent near $0$ is 
more complicated in lower dimensions.

If for some $\chi \in C_c^\infty(\Real^d)$ the limit (\ref{eq:0beh}) {\em fails}
to exist, we will say that $0$ is a  ``non-eigenvalue
resonance'' of $-\Delta +V$.  We shall use this notation even if $0$ is simultaneously 
an eigenvalue of $-\Delta +V$, so that it is
possible for $0$ to be both an eigenvalue resonance and a non-eigenvalue
resonance. (We note that here is an awkwardness that arises from our 
convention that the square roots of eigenvalues in the closure of the 
physical space are resonances.  Had we not
adopted that convention, we could just say that $0$ is a resonance in this 
case.  Many writers do choose this other convention.  However,
our convention is more convenient for some other purposes.)  The non-eigenvalue
resonances at $0$ correspond to elements of the null space of $-\Delta +V$ which
are bounded (if $d=2$) or decaying at infinity (if $d=4$), but which are not
in $L^2(\Real^d)$.

In dimension $d=4$, some of our techniques do not work if $0$ is a 
non-eigenvalue
resonance
of $-\Delta +V$.  Hence for some results in dimension $4$ 
we shall need to assume the following hypothesis.
\begin{hyp}\label{hy:H1}  Let $\mathcal{P}_0$ denote projection onto the 
$L^2$ null space of $-\Delta+V$, if any.  Then for any $\chi \in C_c^\infty(\Real^4)$, $$\lim_{\epsilon \downarrow 0} \chi \left((-\Delta +V+\epsilon^2)^{-1}
-\epsilon^{-2}\mathcal{P}_0\right) \chi$$
exists.
\end{hyp}
We note that if this limit exists for one 
nontrivial $\chi_0\in C_c^\infty(\Real^4)$ with $\chi_0 V=V$, then it exists for
any $\chi\in C_c^\infty(\Real^4)$.  Moreover, this hypothesis is true 
generically.

Finally, we shall need a notion of the multiplicity of $0$ as a resonance
of $-\Delta +V$.  If $0$ is not a non-eigenvalue resonance, this is straightforward,
and is the dimension of the $L^2$ null space of $-\Delta +V$.  However, if
$0$ is a non-eigenvalue resonance, there are several possible notions.  We 
choose the one which is  most convenient for our purposes, but which may
not be the most natural in terms of the dimension of the space of bounded/decaying elements of the null space of $-\Delta +V$.

Let $S(\lambda)$ denote
the scattering matrix of $-\Delta +V$ and let 
$-\mu_1^2 \leq -\mu_2^2\leq ...\leq -\mu_K^2\leq 0$ denote the eigenvalues
of $-\Delta +V$, repeated according to multiplicity.  Our assumptions on
$V$ ensure that there are at most finitely many eigenvalues
and they are  real and non-positive. 
We shall use the heat trace repeatedly in our proofs, and we 
recall one expression for it here.  The Birman-Krein formula
tells us, for $t>0$, 
\begin{equation}\label{eq:b-k}
\tr(e^{t(\Delta -V)}-e^{t\Delta})= \frac{1}{2\pi i}\int_0^\infty \tr 
\left(S(\lambda)^{-1}
\frac{d}{d\lambda} S(\lambda)\right)e^{-t \lambda^2}d\lambda +\sum _{k=1}^Ke^{t\mu_k^2}+\beta(V,d)
\end{equation}
see \cite{CdV,gu}, \cite[Section 3.8]{dy-zw}, and references therein.
Here $\beta(V,d)$ is $0$ whenever $0$ is not a non-eigenvalue resonance of 
$-\Delta +V$.  However, the converse is not true
and the exact behavior is rather subtle; see \cite{b-d-o,b-g-d}.  
Our notion of multiplicity of $0$ as a resonance would better be 
called a normalized or weighted multiplicity, but 
in the interest of brevity we shall just refer to it as 
multiplicity.  We define
the multiplicity of $0$ as a resonance of $-\Delta +V$ to be
$$\beta(V,d)+\dim\{ f\in L^2(\Real^d): \;(-\Delta+V)f=0\}.$$
Note  that if $0$ is an eigenvalue of $-\Delta+V$ 
then it makes a
contribution  in (\ref{eq:b-k}) to the sum over the eigenvalues.




\section{Multiplicities of poles of the resolvent and scattering determinant}
\label{s:multiplicities}

In this section we clarify the relationship between the poles of the
resolvent and the poles of the determinant of the scattering matrix.  
This is well-known in odd dimensions, but is a bit more subtle in even 
dimensions.  The discussion here is taken from \cite{ch-hi4}, though modified
to reflect the fact that
we need only a somewhat simplified version for our specific case of the 
Schr\"odinger operator on $\Real^d$.

Let $d$ be even.  We consider here the 
case of poles of the resolvent which lie on $\Lambda$--that
is, all of the poles except the (possible) pole at $0$.
We define the notion of the multiplicity $\mu_{R_V}$  of 
the pole of the resolvent as follows. Given $\lambda_0\in \Lambda$, define $\gamma_{\lambda_0}$ to be 
a small positively oriented
circle centered at $\lambda_0$ that contains no poles of 
the resolvent except, possibly, a pole at $\lambda_0$. Here we locally 
identify a subset of $\Lambda$ with a subset of the complex plane.
Define, for $\lambda_0
\in \Lambda$,
\begin{equation}
\label{eq:multresolvent}
\mu_{R_V}(\lambda_0) \defeq \rank  \int_{\gamma_{\lambda_0}} 
 R_V(\lambda) 2 \lambda d\lambda = \rank \int_{\gamma_{\lambda_0}}
 R_V(\lambda) d\lambda.
\end{equation}

We shall also consider poles of the determinant of the scattering matrix,
 a scalar function on $\Lambda$.  Following \cite{ch-hi4},
let $f$ be a (scalar) function meromorphic on $\Lambda$.  
If $f(\lambda_0)=0$, define $\msca (f,\lambda_0)$ to be the 
multiplicity of $\lambda_0$ as a zero of $f$.  If $f$  has a pole at $\lambda_0$, define $\msca (f,\lambda_0)$ to be 
minus the order of the pole of $f$ at $\lambda_0$.  If $\lambda_0$ is neither a pole nor a zero of $f$, 
set $\msca (f,\lambda_0)=0$.  Thus $\msca(f, \cdot)$ is positive at zeros and negative at poles.  It should be thought of as measuring the order of vanishing
of the function $f$ at $\lambda_0$.

From \cite[Theorem 4.5]{ch-hi4}, if $V\in L^\infty_c(\Real^d;\Real)$, then
for $\lambda_0\in \Lambda$,
\begin{equation}\label{eq:multrelation}
\mu_{R_V}(\lambda_0) - \mu_{R_V}(\overline{\lambda_0})
=-\msca(\det S(\lambda),\lambda_0),
\end{equation}  where $\overline{\lambda} = |\lambda| e^{-i \arg \lambda}$.
Thus the determinant of the scattering matrix does not necessarily have a 
pole at each pole of the resolvent.  However, if the determinant of the scattering matrix
has a pole at $\lambda_0$, then the resolvent must have a pole at $\lambda_0$.

The following lemma and its corollary will be used in the proofs of 
Theorem \ref{thm:infinitelymany} and Corollary \ref{cor:infinitediff}.
\begin{lemma}\label{l:multiplicities}
 Let $V_1,\; V_2 \in L^\infty_c(\Real^d;\Real)$, and for $j=1,2$
let
$S_j$ denote the scattering matrix for the operator $-\Delta +V_j$.
Suppose $\det S_1(\lambda)/\det S_2(\lambda)$ is analytic on all of $\Lambda$.
Let $\mu_{R_j}(\lambda)$ denote the quantity (\ref{eq:multresolvent})
for the operator $-\Delta +V_j$. Then, for $\tau_0>0$, $\theta_0\in \Real$,
$$\mu_{R_1}(\tau_0 e^{i(\theta_0+k\pi)})-\mu_{R_2}(\tau_0 e^{i(\theta_0+k\pi)})
= \mu_{R_1}(\tau_0 e^{i\theta_0})-\mu_{R_2}(\tau_0 e^{i\theta_0}),\; \text{if}\;k\in \Integers$$
and 
$$\mu_{R_1}(\tau_0 e^{i(-\theta_0+k\pi)})-\mu_{R_2}(\tau_0 e^{i(-\theta_0+k\pi)})
= \mu_{R_1}(\tau_0 e^{i\theta_0})-
\mu_{R_2}(\tau_0 e^{i\theta_0}),\; \text{if}\; k\in \Integers.$$
\end{lemma}
\begin{proof}
Note that our assumption on $\det S_1(\lambda)/\det S_2(\lambda)$ implies that
this ratio is non-vanishing on all of $\Lambda$.  Hence, for all 
$\lambda_0 \in \Lambda$, 
\begin{equation}\label{eq:mscasame}
\msca(\det S_1(\lambda), \lambda_0)=
 \msca(\det S_2(\lambda), \lambda_0).
\end{equation}

 By (\ref{eq:multrelation})
and (\ref{eq:mscasame}),
\begin{equation}
\label{eq:mudiff1}
\mu_{R_1}(\tau e^{i\theta})-\mu_{R_2}(\tau e^{i\theta})= 
\mu_{R_1}(\tau e^{-i\theta})-\mu_{R_2}(\tau e^{-i\theta})
\end{equation}
for any $\tau>0$, $\theta\in \Real$.
Since $R_j(\lambda)= R^*_j(e^{i\pi} \overline{\lambda})$, $\mu_{R_j}(\lambda)=
\mu_{R_j}(e^{i\pi} \overline{\lambda})$ and 
\begin{equation}\label{eq:step2}
\mu_{R_j}( \tau e^{\pm i \theta})= \mu_{R_j}(\tau e^{i(\pi \mp\theta)}).
\end{equation}
Combining this with (\ref{eq:mudiff1}) gives
$$
\mu_{R_1}(\tau e^{i(\pi\pm \theta)})-\mu_{R_2}(\tau e^{i(\pi \pm \theta)})
= \mu_{R_1}(\tau e^{i\theta})-\mu_{R_2}(\tau e^{i\theta}).$$
Applying (\ref{eq:mudiff1}) again, this gives
\begin{align}\label{eq:step4}\nonumber
\mu_{R_1}(\tau e^{-i(\pi\pm \theta)})-\mu_{R_2}(\tau e^{-i(\pi \pm \theta)}) & 
=
\mu_{R_1}(\tau e^{i(\pi\pm \theta)})-\mu_{R_2}(\tau e^{i(\pi \pm \theta)}) \\
& 
= \mu_{R_1}(\tau e^{i\theta})-\mu_{R_2}(\tau e^{i\theta}).
\end{align}

Using (\ref{eq:mudiff1}), (\ref{eq:step2}), and (\ref{eq:step4}) 
inductively proves the lemma.
\end{proof}

Lemma \ref{l:multiplicities} has the following corollary
as a special case.  This corollary will be used in the proof of
Theorem \ref{thm:infinitelymany}.
\begin{cor} \label{c:negativeeigenvalues}
 Let $V\in L^\infty_c(\Real^d;\Real)$ and denote the associated scattering
matrix $S(\lambda)$.  If $\det S(\lambda)$ is analytic on all
of $\Lambda$ but $-\Delta +V$ has a negative eigenvalue, then 
the meromorphic continuation of the  resolvent of $-\Delta +V$
has infinitely many poles.  In particular, in this case if $-\rho^2 $ is an 
eigenvalue of $-\Delta +V$ with multiplicity $m_0>0$ and if $\rho>0$, then 
$e^{i\pi(k+1/2)}\rho \in \Lambda$ is a pole of $\chi R_V(\lambda)\chi$
of multiplicity $m_0$ for every 
$k\in \Integers$.
\end{cor}
\begin{proof}
In Lemma \ref{l:multiplicities}, take $V_1=V$, $V_2\equiv 0$, 
$\tau_0=\rho$ and $\theta_0=\pi/2$.
\end{proof}
\section{High energy behavior of the determinant of 
the scattering matrix on the 
real line}
\label{s:asymptotic}

The proof of Theorem \ref{thm:samepoles} which we shall give uses Theorem
\ref{thm:derivasmpt}, a result about the large $\lambda$ behavior of the determinant of
the scattering matrix on the positive real axis.    A stronger result
 is well known for smooth potentials $V$ \cite{CdV,gu,pop} and \cite[Theorem 9.2.12]{yafaevmathematical}, 
or even, if $d=3$
for some potentials
 with less regularity \cite{jensenasymptotics}, but still with 
more regularity than $L^\infty$. 
A related but slightly different 
 result for dimensions $2$ and $3$ 
and $V\in L^\infty_c(\Real^d;\Real)$ is \cite[Theorem 3.12]{yafaev}
or \cite[Theorem 9.1.14]{yafaevmathematical}. 
However, we are unaware of a result
which is valid  in all dimensions 
$d\geq 2$ for the class of potentials which we consider here.  
The parity of the dimension is not important here.  Hence 
in this section the dimension $d$ is allowed to be even or odd, but
we always assume $d\geq 2$.
\begin{thm}\label{thm:derivasmpt} Let $d\geq 2$ be even or odd.
Let $V\in L^{\infty}_c(\Real^d;\Real)$, and let $S(\lambda)$ be the associated
scattering matrix. Then for $\lambda \in \Real_+$ (i.e., $\arg \lambda=0$),
$$\frac{1}{i}\frac{\frac{d}{d\lambda} \det S(\lambda)}{\det S(\lambda)}
= -(d-2)c_d \int V(x)dx \lambda^{d-3}+O(\lambda^{d-7/2})\; \text{as $\lambda \rightarrow \infty$}.$$
Here $c_d=\pi (2\pi)^{-d} \vol(\Sphere^{d-1})$.
\end{thm}
Note that when $d=2$, the coefficient of $\lambda^{d-3}$ is $0$.

\subsection{Reduction of the proof of Theorem \ref{thm:derivasmpt}
to Lemma \ref{l:B2est}}
The proof uses an explicit expression for the scattering
matrix (e.g. \cite[(8.1)]{yafaevbook})
\begin{equation}\label{eq:scattmatrix}
S(\lambda) = I-2\pi i \Gamma_0(\lambda) (V-VR_V(\lambda)V)\Gamma_0(\lambda)^*,
\; \text{if $\lambda \in (0,\infty)$}
\end{equation}
where 
$$\left(\Gamma_0(\lambda)f\right)(\omega)= 2^{-1/2} (2\pi)^{-d/2}\lambda^{(d-2)/2}
\hat{f}(\lambda \omega),\; \omega \in \Sphere^{d-1}$$
and $\hat{f}(\xi)=\int e^{-i x \cdot \xi}f(x)dx$.
Consistent with our notation elsewhere, for 
$\lambda>0$  $R_V(\lambda)=(-\Delta+V-(\lambda +i0)^2)^{-1}$.

Let $\chi\in C_c^{\infty}(\Real^d)$.
By \cite[Theorem 16.1]{k-k}, with $R_0(\lambda)=(-\Delta -\lambda^2)^{-1}$ 
when $\Im \lambda>0$, 
\begin{equation}\label{eq:resbd}
\left\| \frac{d^j}{d\lambda^j}\chi R_0(\lambda) \chi \right\| \leq C_j \lambda^{-1} 
\; \text{when $\lambda \in (0,\infty)$,  $j\in \Natural_0$}.
\end{equation}
Here the value of $C_j$ depends on $\chi$ as well as on $j$.

Since $$\| \Gamma_0(\lambda )\chi f\|^2 = \pi^{-1} \Im  \langle \chi R_0(\lambda)\chi f,f\rangle$$
\cite[(10.1)]{yafaevbook}, we have
\begin{equation}\label{eq:gammaest}\left\|  \Gamma_0 (\lambda) \chi \right\| 
\leq C \lambda^{-1/2},\; 
\left\| \frac{d}{d\lambda }  \Gamma_0 (\lambda) \chi \right\| 
\leq C \lambda^{-1/2} 
\; \text{when $\lambda \in (0,\infty)$}.
\end{equation}
With $\| \cdot \|_{HS}$ denoting the Hilbert-Schmidt norm
\begin{equation}\label{eq:HSnorm}
\| \Gamma_0(\lambda) \chi\|_{HS}^2  =  \frac{1}{2(2\pi)^d}\int_{x\in \supp \chi}
\int _{\omega \in \Sphere^{d-1}} \left|\lambda^{(d-2)/2} e^{-i\lambda x\cdot \omega} \chi(x)\right|^2
d\sigma_\omega dx
\leq C \lambda^{d-2},\; \lambda 
\in (0,\infty)
\end{equation}
where $d\sigma$ denotes the density on $\Sphere^{d-1}$.
Likewise,
 $\| \frac{d}{d\lambda} \Gamma_0(\lambda) \chi\|_{HS}^2\leq C \lambda^{d-2}$ for $\lambda \in (1,\infty)$.

For an operator $A$ depending on $\lambda$,
 denote by $\dot{A}(\lambda)$ the derivative of $A(\lambda)$ with respect to $\lambda$.

Recall that 
for $\lambda \in (0,\infty)$, $S^*(\lambda)S(\lambda)=I$, so that
\begin{equation}
\tr[ S^*(\lambda) \dot{S}(\lambda)]= \tr[ S^{-1}(\lambda) \dot{S}(\lambda)]
= \frac{\frac{d}{d\lambda}\det S(\lambda)}{\det S(\lambda)},\; \lambda \in 
(0,\infty).
\end{equation}

Set
\begin{equation}\label{eq:Bs}
B_1(\lambda) = -2\pi i \Gamma_0(\lambda) V \Gamma_0^*(\lambda)\; \text{and}\;
B_2(\lambda) = 2\pi i \Gamma_0(\lambda) V  R_0(\lambda) V \Gamma_0^*(\lambda).
\end{equation}
\begin{lemma}\label{l:tracewithBs}Let $B_1$, $B_2$ be as defined in (\ref{eq:Bs}).
Then for $\lambda \in (0,\infty)$, 
$$\tr (S^*(\lambda)\dot{S}(\lambda))= 
\tr \left( \dot{B}_1(\lambda)+  \dot{B}_2(\lambda)+ B_1^*(\lambda) \dot{B}_1(\lambda)
\right) +O(\lambda^{d-4})\; \text{ when $\lambda \rightarrow \infty$}.
$$
\end{lemma}
\begin{proof}
We use the expression (\ref{eq:scattmatrix}) for the scattering matrix.  
Let $\chi\in C_c^{\infty}(\Real^d)$ satisfy $\chi V=V$.  Since, for large $\lambda\in 
\Real_+$,  $\chi R_V(\lambda) \chi =
\chi R_0(\lambda) \chi \sum _{j=0}^{\infty}(-VR_0(\lambda) \chi)^j,$
using the bounds (\ref{eq:resbd}-\ref{eq:HSnorm}) we see that 
$$\left \| \frac{d^j}{d\lambda^j}\left[ S(\lambda)- (I+B_1(\lambda)+ B_2(\lambda))
\right] \right \|_{\tr}
\leq C \lambda^{d-4} \; 
\text{for $\lambda \in \Real, \; \lambda \geq 1,\; j=0,\; 1$}
$$
and
$$\left\| \frac{d^j}{d\lambda^j}\left[ S(\lambda)- (I+B_1(\lambda)+ B_2(\lambda))
\right] \right\|_{L^2 \rightarrow L^2}
\leq C \lambda^{-3} \; \text{for $\lambda \in \Real, \; \lambda \geq 1,\; j=0,\; 1$}.
$$
Moreover, 
\begin{equation}\label{eq:B2easyestimate}
\|B_2(\lambda) \|_{L^2\rightarrow L^2} \leq C \lambda^{-2}, \;
\| B_2(\lambda)\|_{\tr} \leq C\lambda^{d-3}.
\end{equation}
Hence, by (\ref{eq:scattmatrix})
$$ \tr (S^*(\lambda)\dot{S}(\lambda))
= \tr [ (I + B_1^*(\lambda) )(\dot{B}_1(\lambda)+\dot{B}_2(\lambda))]
+ O(\lambda^{d-4}).$$
Similarly, 
\begin{align*}
\left| 
\tr[B_1^*(\lambda) \dot{B}_2(\lambda)]\right| & 
\leq \| B_1^*(\lambda) \|_{\tr} \| \dot{B}_2(\lambda)\|_{L^2\rightarrow L^2}\\ & 
\leq C \lambda^{d-2} \lambda^{-2}
\end{align*} 
for some constant $C$, since $\| \dot{B}_2\|_{L^2\rightarrow L^2}= O(\lambda^{-2})$.
\end{proof}

\begin{lemma}\label{l:trB1} Let $B_1$ be as defined in (\ref{eq:Bs}).  Then
$$\tr[\dot{B}_1(\lambda)]= -i c_d (d-2)\left(\int V(x) dx\right) \lambda^{d-3}$$
where $c_d=\pi(2\pi)^{-d} \vol(\Sphere^{d-1})$.
\end{lemma} 
\begin{proof}Here we can evaluate the trace of $\dot{B}_1$  as the integral of
 its
Schwartz kernel over the diagonal.  Hence, with $d\sigma$ denoting the 
usual measure on $\Sphere^{d-1}$,
\begin{align*} & 
\tr[\dot{B}_1(\lambda)]\\ & = -\frac{\pi i}{(2\pi)^d} \int_{\omega \in \Sphere^{d-1}}\int_{x\in \Real^d}
e^{i\lambda x\cdot(\omega-\theta)}\restrict_{\theta=\omega} V(x) 
\left( (d-2) \lambda^{d-3} + 
\lambda^{d-2} i x\cdot(\omega -\theta)\right)\restrict_{\theta=\omega}dx \; 
d\sigma_{\omega}
\\ & = -\frac{\pi i(d-2)}{(2\pi)^d} \lambda^{d-3} \int_{\omega \in \Sphere^{d-1}}\int_{x\in \Real^d}
 V(x) dx \; d\sigma_{\omega}
\\ & = -\frac{\pi i(d-2)}{(2\pi)^d} \lambda^{d-3}\vol(\Sphere^{d-1}) \int V(x) dx.
\end{align*}
\end{proof}

\begin{lemma}\label{l:tr2B1s} Let $B_1$  be as defined in (\ref{eq:Bs}) and $\lambda \in (0,\infty)$.  Then
$\Im \tr[ B_1^*(\lambda) \dot{B_1}(\lambda)]=0.$
\end{lemma}
\begin{proof}
We have 
\begin{equation*}
\tr[B_1^*(\lambda) \dot{B_1}(\lambda)]
= 4\pi^2 \tr \{ \Gamma_0(\lambda) V \Gamma_0^*(\lambda)
[ \dot{\Gamma}_0(\lambda) V\Gamma_0^*(\lambda)+
 \Gamma_0(\lambda) V \dot{\Gamma}_0^*(\lambda)]\}.
\end{equation*}
We use the cyclicity of the trace to write this as 
\begin{equation*}
\tr[B_1^*(\lambda) \dot{B_1}(\lambda)]
= 4\pi^2 \tr\{ \Gamma_0(\lambda) V [\Gamma_0^*(\lambda)\dot{\Gamma}_0(\lambda)+
\dot{\Gamma}_0^*(\lambda) \Gamma_0(\lambda)] V \Gamma_0^*(\lambda)\}. 
\end{equation*}
Since this is the trace of a self-adjoint operator, its imaginary part is $0$.
\end{proof}

Note that for $\lambda \in (0,\infty)$, $S(\lambda)$ is unitary
so that $[\frac{d}{d\lambda}\det S(\lambda)]/\det S(\lambda)$ is pure imaginary.
Thus, by Lemmas \ref{l:tracewithBs}, \ref{l:trB1}, 
and \ref{l:tr2B1s}, 
to prove Theorem \ref{thm:derivasmpt} it remains only to estimate the 
 imaginary part of $\tr [ \dot{B}_2(\lambda)]$.
Hence the proof of Theorem \ref{thm:derivasmpt} is completed by 
the following lemma.
\begin{lemma}\label{l:B2est}
For $\lambda \in (0,\infty)$, $\Im[\tr(\dot{B}_2(\lambda) )]=O(\lambda^{d-7/2})$
as $\lambda \rightarrow \infty.$
\end{lemma}
The proof of this lemma is the content of the next subsection.

\subsection{Proof of Lemma \ref{l:B2est}}
In this section we study $\Im \tr B_2(\lambda)$, and 
prove Lemma \ref{l:B2est}.
Writing the integral as a trace of the Schwartz kernel over the diagonal,
$$\tr B_2(\lambda)= \pi i(2\pi)^{-d} \lambda^{d-2}\int _{\omega \in \Sphere^{d-1}}
\int _{x\in \Real^d} e^{-i\lambda x\cdot \omega} V(x) W(x,\omega,\lambda) dx d\sigma_\omega$$
where 
$W(x,\omega,\lambda)= (R_0(\lambda)(V(\bullet)e^{i\lambda \omega \cdot \bullet })(x)$.
Using Parseval's formula to evaluate the $x$ integral, we get 
$$\tr B_2(\lambda)=  \frac{i}{2} (2\pi)^{-2d+1} \lambda^{d-2}
\lim_{\epsilon \downarrow 0}\int _{\omega \in \Sphere^{d-1}}
\int _{\eta' \in \Real^d } \frac{|\hat{V}(\eta' -\lambda \omega)|^2}{|\eta'|^2-(\lambda+i\epsilon)^2} d\eta' d\sigma_{\omega}.$$
Making the substitution $\eta'=\lambda \eta +
\lambda \omega $ gives
\begin{equation}\label{eq:B2expr}
\tr B_2(\lambda)= \frac{ i}{2} (2\pi)^{-2d+1} \lambda^{2(d-2)}
\lim_{\epsilon \downarrow 0}\int _{\omega \in \Sphere^{d-1}}
\int _{\eta \in \Real^d } 
\frac{|\hat{V}(\lambda \eta)|^2}{|\eta+\omega|^2-1-i\epsilon}
d\eta \; d\sigma_{\omega}.
\end{equation}  

We shall use the following lemma to justify some manipulations in our 
estimate of the derivative of $\Im \tr B_2(\lambda)$.  Since we shall want to 
estimate the size of the derivative of the
imaginary part of (\ref{eq:B2expr}), we take (twice) the 
average of the distributions $\lim_{\epsilon \downarrow 0}
1/(|\eta+\omega|^2-1\pm i \epsilon)$ in the 
following lemma and in Lemma \ref{l:gprop}.
\begin{lemma}\label{l:nodeltas} Let $\psi \in C_0^\infty(\Real)$ be supported in a neighborhood of $0$, and let $\phi \in C_0^\infty(\Real^d)$.  Then
\begin{equation}\label{eq:rhoat0}
\lim_{\delta \downarrow 0} \lim_{\epsilon \downarrow 0} 
\int_{\omega \in \Sphere^{d-1}} \int _{\eta \in \Real^d}
\sum _{\pm} \frac{1}{ |\eta + \omega|^2-1\pm i\epsilon}
\psi (|\eta|/\delta)\phi(\eta) d\eta d\sigma_\omega=0
\end{equation}
and 
\begin{equation}\label{eq:rhoat2}
\lim_{\delta \downarrow 0} \lim_{\epsilon \downarrow 0} 
\int_{\omega \in \Sphere^{d-1}} \int_{\eta \in \Real^d}
\sum_{\pm} \frac{1}{ |\eta + \omega|^2-1\pm i\epsilon}
\psi ((|\eta|-2)/\delta)\phi(\eta) d\eta d\sigma_\omega=0.
\end{equation}
\end{lemma}
\begin{proof}
We give the proof of (\ref{eq:rhoat2}), as the proof of (\ref{eq:rhoat0})
is similar.

Using polar coordinates $\eta =\rho \theta$ and \cite[(3.2.13)]{ho1},
\begin{multline}
\lim_{\epsilon \downarrow 0} 
\int_{\omega \in \Sphere^{d-1}} \int_{\eta \in \Real^d}
\sum _{\pm} \frac{1}{ |\eta + \omega|^2-1\pm i\epsilon}
\psi ((|\eta|-2)/\delta)\phi(\eta) d\eta d\sigma_\omega \\
=-2\int_{\omega \in \Sphere^{d-1}} \int_{\theta \in \Sphere^{d-1}}\int_0^\infty 
\log|\rho+2\theta \cdot \omega|\frac{\partial}{\partial \rho}
[\rho^{d-2} 
\psi ((\rho-2)/\delta)\phi(\rho \theta)] d\rho d\sigma_\theta d\sigma_\omega.
\end{multline}
The integrand is absolutely integrable, and we may change the order of 
integration as convenient.  We would like to understand the behavior of
\begin{equation}
\label{eq:sphereint}
 \int_{\omega \in \Sphere^{d-1}} \log|\rho+2\theta \cdot \omega|d\sigma_\omega
\end{equation}
near $\rho=2$.  If $\rho$ is near $2$,
the integrand in (\ref{eq:sphereint}) is smooth away from a neighborhood
of $\omega=-\theta$.  Therefore,
the integral over $\{ \omega \in \Sphere^{d-1}: \; |\omega+\theta|\geq
\epsilon'>0\}$
 results in a function depending smoothly on 
$\rho$ near $\rho=2$.  Thus we concentrate on a neighborhood
 of $\omega=-\theta$ in $\Sphere^{d-1}$.  The function 
$\theta\cdot \omega$ has a nondegenerate critical point at $\theta=-\omega$.
By first integrating over level sets in $\omega\in \Sphere^{d-1}$
 of $\theta \cdot \omega$ and then using the Morse lemma,
we may then introduce a variable $s$ so that 
$$
\int_{\{\omega \in \Sphere^{d-1}: |\theta \cdot \omega +1|<\epsilon'\} } 
\log|\rho+ 2\theta \cdot \omega|d\sigma_\omega
= \int_{|s|<\epsilon''} \log |\rho -2 + s^2|h(s)ds$$
for some smooth function $h$.  But then 
$$\int_{|s|<\epsilon''} \log |\rho -2 + s^2|h(s)ds - 
\int_{|s|<\epsilon''} \log |2 -2 + s^2|h(s)ds 
= \int_{|s|<\epsilon''}  \log \left| \frac{ \rho -2+s^2}{s^2} \right| h(s)ds.
$$

By a change of variables, for $\rho \not =2$, $\rho $ near $2$,
\begin{multline}\label{eq:covsqrt}
\int_{|s|<\epsilon''}  \log \left| \frac{ \rho -2+s^2}{s^2} \right| h(s)ds \\
= \sqrt{|\rho-2|}\int _{|s|<\epsilon''/\sqrt{|\rho-2|}}
\log \left|  \frac{ \sgn(\rho-2)+s^2}{s^2}\right| h(s \sqrt{|\rho-2|})ds.
\end{multline}
Since $\log \left[  s^{-2} |\sgn(\rho-2) +s^2|\right]=O(s^{-2})$ when 
$|s|\rightarrow \infty$, the integral on the right hand side of 
(\ref{eq:covsqrt}) is bounded independently of $\rho$ near $2$, hence
\begin{equation}\label{eq:diffest}
\int_{\omega \in \Sphere^{d-1}} \log|\rho+2\theta \cdot \omega|d\sigma_\omega
- \int_{\omega \in \Sphere^{d-1}} \log|2+2\theta \cdot \omega|d\sigma_\omega = 
O(\sqrt{|\rho-2|}).
\end{equation}
But then
\begin{align} \label{eq:limit} \nonumber
& 
\iint_{\omega,\theta \in \Sphere^{d-1}} 
\int_0^\infty 
\log|\rho+2\theta \cdot \omega|\frac{\partial}{\partial \rho}
[\rho^{d-2} 
\psi ((\rho-2)/\delta)\phi(\rho \theta)] d\rho d\sigma_\theta d\sigma_\omega
\nonumber \\ & 
= \nonumber
\iint_{\omega,\theta \in \Sphere^{d-1}} 
\int_0^\infty 
\left[ \log|\rho+2\theta \cdot \omega|- \log | 2+2\theta \cdot \omega|+
\log |2+2\theta \cdot \omega|\right] \\ & \nonumber \hspace{40mm} \times
\frac{\partial}{\partial \rho}
[\rho^{d-2} 
\psi ((\rho-2)/\delta)\phi(\rho \theta)] d\rho d\sigma_\theta d\sigma_\omega
 \\ & =
\iint_{\omega,\theta \in \Sphere^{d-1}} 
\int_0^\infty 
\left[ \log|\rho+2\theta \cdot \omega|- \log | 2+2\theta \cdot \omega |\right]
\frac{\partial}{\partial \rho}
[\rho^{d-2} 
\psi ((\rho-2)/\delta)\phi(\rho \theta)] d\rho d\sigma_\theta d\sigma_\omega
\end{align}
since
$$ 
\iint_{\omega ,\theta\in \Sphere^{d-1}} \int_0^\infty 
\log | 2+2\theta \cdot \omega|
\frac{\partial}{\partial \rho}
[\rho^{d-2} 
\psi ((\rho+2)/\delta)\phi(\rho \theta)] d\rho d\sigma_\theta d\sigma_\omega=0$$
by the support properties of $\psi$ and using that 
$\log|2+2\theta \cdot \omega|$ is independent of $\rho$.
Finally, using (\ref{eq:diffest}) shows that the final
integrand in (\ref{eq:limit}) is $O(\delta^{-1/2})$ in $L^\infty$, 
$O(\delta^{1/2})$ in $L^1$.  Hence the limit 
as $\delta \downarrow 0$ is $0$ as claimed.
\end{proof}

In practice to understand (\ref{eq:B2expr})
 we shall want to evaluate the $\omega$ integrals first, and interchange
the order of the limit and the integral over $\eta\in \Real^d$.
We shall use the following two lemmas to rigorously justify this and 
to understand the limit as $\epsilon \downarrow 0$ of the $\omega$ integral.
\begin{lemma}\label{l:gcont1} Let $h\in C_c^\infty ( (-1,1);\Real)$.  Then there is a $C>0$ so that for $0<\epsilon <1$, $t>0$, 
\begin{equation}\label{eq:h1est1}
\left| \Re \int \frac{1}{t^2+2ts-i\epsilon}h(s) ds\right| \leq 
C \frac{1}{t(t+1)}(1+ |\log |\epsilon/(t^2-2t)||).
\end{equation}
Moreover, if $t>0$,  then setting 
$$g(t)=\lim_{\epsilon\downarrow 0} \Re \int \frac{1}{t^2+2ts-i\epsilon}h(s) ds$$
the function $g $ is continuous function on $(0,\infty)$ and satisfies
$|g(t)|\leq C/(t(t+1))$.
\end{lemma}
\begin{proof} If $t\geq 2+\delta>2$ the result is immediate.  Hence we 
restrict ourselves to $0<t<4$.

   Use $\log z $ to denote the principal branch of the logarithm.  
Then for $\epsilon >0$
\begin{align}\label{eq:rewrite1}
\Re \int \frac{1}{t^2+2ts-i\epsilon}h(s) ds & = \Re \frac{1}{t}\int 
\frac{1}{t+2s-i\epsilon/t}h(s) ds\nonumber \\ & = \Re \frac{1}{t}\int 
\frac{1}{2u-i\epsilon/t}h(u-t/2) du \nonumber\\ & 
= -  \Re \frac{1}{2t}\int 
\log(2u-i\epsilon/t)h'(u-t/2) du
\end{align}
with the change of variable $u= s+t/2$.
The estimate (\ref{eq:h1est1}) follows from this and the support and smoothness
properties of $h$ and integrability of $\log (2u-i\epsilon/t)$ and 
$\log |2u|$.
From (\ref{eq:rewrite1}) we can see that the limit 
as $\epsilon \downarrow 0$ exists,
is given by 
$$g(t)=- \frac{1}{2t}\int \log |2u| h'(u-t/2)du,$$ 
and is continuous for positive $t$.
Moreover, for $0<t<4,$ $|g(t)| \leq C/t$.
\end{proof}

This next lemma proves a similar result.  Notice  a difference
between this lemma and the previous one comes
from the denominator in the integrand having a stationary point at $s=0$.
\begin{lemma}\label{l:gcont2}
Let $\tilde{h}\in C_c^{\infty}( (-1/4, 1/4);\Real)$.  Then there is
a $C>0$ so that for $0<\epsilon<1$ and $t>0$
\begin{equation}\label{eq:h2est1}
\left| \Re \int \frac{1}{t^2-2t\sqrt{1-s^2}-i\epsilon}\tilde{h}(s) ds\right| \leq 
 \frac{C}{t|t^2-4|^{1/2}}\left(1+ \left|\log (\epsilon/|t^2-2t|)\right|\right).
\end{equation}
Moreover, if $t>0$, $t\not =2$,  then setting 
$$g(t)=\lim_{\epsilon\downarrow 0} \Re \int \frac{1}{t^2-2t\sqrt{1-s^2}-i\epsilon}
\tilde{h}(s) ds$$
the function $g $ is continuous function on $(0,2)\cup (2,\infty)$ 
and satisfies
$|g(t)|\leq C/(t|t^2-4|^{1/2})$.
\end{lemma}
\begin{proof}
The lemma is immediate if $t\geq 2+\delta>2$, so we shall assume $0<t<4$, 
$t\not =2$.

By the Morse Lemma, there is a smooth function $u=u(s)$ 
defined for $|s|\leq 1/4$  so that 
$\sqrt{1-s^2}=1-u^2$, and $s$ can be written as a smooth function of $u$.
Then
\begin{equation}\label{eq:ml}
\int \frac{1}{t^2-2t\sqrt{1-s^2}-i\epsilon}
\tilde{h}(s) ds= \int \frac{1}{t^2-2t(1-u^2)-i\epsilon}
h(u)du
\end{equation}
where $h(u)=\tilde{h}(s(u))\frac{ds}{du}$.
The integral in  (\ref{eq:ml}) can be rewritten
\begin{align*}
\int \frac{1}{t^2-2t(1-u^2)-i\epsilon}h(u)du & 
= \frac{1}{2t|t-2|}\int \frac{1}{\frac{1}{2}\sgn (t-2) +\frac{u^2}{|t-2|}-i\frac{\epsilon}
{2t|t-2|}}h(u)du\\
& = \frac{1}{2t|t-2|^{1/2}}\int \frac{1}{\frac{1}{2}\sgn (t-2) +u^2-i\frac{\epsilon}
{2t|t-2|}}h(u|t-2|^{1/2})du.
\end{align*}
Let  $\alpha =\alpha(\epsilon, t)$ be such that
 $\alpha^2= -\frac{1}{2}\sgn (t-2)+i\frac{\epsilon}{2t|t-2|}$.  
Then, with $\log z$ denoting the principal branch of the logarithm
defined on $\Complex \setminus(-\infty,0]$,
\begin{align*}
&  \Re \frac{1}{2t|t-2|^{1/2}}\int \frac{1}{\frac{1}{2}\sgn (t-2) +u^2-i\frac{\epsilon}
{2t|t-2|}}h(u|t-2|^{1/2})du\\ & 
= \Re \frac{1}{4t\alpha |t-2|^{1/2}}
\int \left(\frac{1}{u-\alpha}-\frac{1}{u+\alpha}\right)h(u|t-2|^{1/2})du\\
& = \Re \frac{-1}{4t\alpha }
\int (\log(u-\alpha)-\log(u+\alpha))h'(u|t-2|^{1/2})du.
\end{align*}
Now we use that $|\log(u\pm \alpha)| \leq (|\log|u\pm \alpha||+\pi/2)$ and
the support properties of $h$ to obtain (\ref{eq:h2est1}).
This also shows the limit as $\epsilon \downarrow 0$ of the 
real part exists for $0<|t-2|<2$, and can be found by interchanging the order
of the limit and integral.  

If $0<t<2$, then when $\epsilon=0$, $\alpha= 1/\sqrt{2}$, and
$$g(t)= \frac{-\sqrt{2}}{4t}\int \log \left| \frac{\sqrt{2}u-1}{\sqrt{2}u+1}\right|
h'(u|t-2|^{1/2})du.$$
This integral yields a continuous function of $t$ for $t\in (0,2)$ which
is bounded as claimed.
If $2<t<4$, then when $\epsilon=0$, $\alpha = i/\sqrt{2}$,
and $g$ is given by
$$g(t)=  \frac{\sqrt{2}}{4t }\Im
\int (\log(u-i/\sqrt{2})-\log(u+i/\sqrt{2}))h'(u|t-2|^{1/2})du.
$$
Since $|\Im (\log(u-i/\sqrt{2})-\log(u+i/\sqrt{2}))| \leq \pi$,
this integral gives a continuous function of $t\in (2,4)$ which
is bounded as claimed.
\end{proof}

The previous two lemmas help to prove the next result.
\begin{lemma}\label{l:gprop} 
There is a $C>0$ so that for $|\eta|\not =0,\; 2$, $0<\epsilon<1$,
\begin{equation}\label{eq:gprop1}\left| \Re \int_{\omega\in \Sphere^{d-1}}
\frac{ 1}{|\eta+\omega|^2-1-i\epsilon} d\sigma_{\omega}\right|
\leq \frac{C(1+|\log (\epsilon/(|\eta|^2-2|\eta|)|)}{|\eta|||\eta|^2-4|^{1/2}}.
\end{equation}
Moreover,
$$\lim_{\epsilon \downarrow 0}\Re \int_{\omega\in \Sphere^{d-1}}
\frac{ 1}{|\eta+\omega|^2-1-i\epsilon}
d\sigma_\omega = g(|\eta|)$$
where $g(t)$ is continuous for $t\in (0,2)\cup(2,\infty)$.  Moreover,
$t\sqrt{|t-2|}g(t)$ is bounded for $t\leq 4$, and there is a constant
$C$ so that $|g(t)|\leq C/t^2$ when $t\geq 4$.
\end{lemma}
\begin{proof}
When $\epsilon>0$, 
by a change of variable of integration (a
rotation) one can see  that
$$\int _{\omega\in \Sphere^{d-1}}
\frac{ 1}{|\eta+\omega|^2-1- i\epsilon}d \sigma_{\omega} 
= \int _{\omega\in \Sphere^{d-1}}
\frac{ 1}{|\eta|^2 + 2 \omega \cdot \eta - i\epsilon}d \sigma_{\omega} $$
 depends on $\eta$ only through $|\eta|$.  Hence, without loss of generality,
to evaluate the integral away from $\eta =0$ we may assume $\eta/|\eta|=
(1,0,...,0)$.
Thus, writing $t$ in place of $|\eta|$,  we wish to bound 
$$\Re \int _{\omega\in \Sphere^{d-1}}
\frac{ 1}{t^2 + 2 \omega_1 t  - i\epsilon}d \sigma_{\omega},$$ 
where $\omega = (\omega_1,...,\omega_{d})$.
 It immediate that the integral is smooth when
$t>2$ and decays at infinity (in $t$) as claimed, even when $\epsilon =0$.
Hence below we may assume $0<t\leq 4$.

Choose a  partition of unity on $\Sphere^{d-1}$, depending only on $\omega_1$, so 
that $1=\chi_+(\omega_1)+
\chi_-(\omega_1) +\chi_m(\omega_1)$, $\chi_{\pm }$ are supported near
$\omega_1=\pm 1$, $\chi_m$ is $0$ in a neighborhood of $\omega_1=\pm 1$, and
all three functions are smooth and real-valued.
 We shall want $\chi_-$ supported 
close to $\omega_1=-1$, say within $2 \delta$, with $\delta>0$ small,
and assume $\chi_m(\omega_1)=0$ if $| \omega_1 \pm 1 |<\delta$.
Since $t> 0$, 
 $$\left| \frac{ 1}{t^2+2\omega_1 t-i\epsilon} \chi_+(\omega_1)\right| \leq
\frac{C}{t}.$$ 
Hence the integral over the support of $\chi_+$ gives a contribution
to (\ref{eq:gprop1}) and
to $g$ which behaves as claimed.

On the support of $\chi_m$ we can use $\omega_1$ as a 
coordinate and can write,  with $t>0$
 and $\delta>0$  sufficiently small,
\begin{align}\label{eq:chim1}
 \Re \int _{\omega\in \Sphere^{d-1}}
\frac{ \chi_m(\omega_1) }{t^2 + 2 \omega_1 t  - i\epsilon}
  d \sigma_{\omega}
= c_d'\Re  \int_{-1+\delta}^{1-\delta}
\frac{ \chi_m(\omega_1)}{t^2 + 2t \omega_1   - i\epsilon}(1-\omega_1^2)^{(d-3)/2}
d\omega_1
\end{align}
where $c_d'$ is a positive constant.  Here we have 
integrated over a space of dimension $d-2$, the 
level surfaces of $\omega_1$ in $\Sphere^{d-1}$.  Now applying Lemma \ref{l:gcont1}
we see the integral over the support of $\chi_m$ gives a contribution to 
 (\ref{eq:gprop1}) and
to $g$ which behaves as claimed.


Now consider the contribution from $\omega_1$ near $-1$.  Writing 
$\omega = (\omega_1,\omega')$ and taking $s^2=|\omega'|^2$,
\begin{equation*}
\Re \int _{\omega\in \Sphere^{d-1}}
\frac{ \chi_{-}(\omega_1) }{t^2 + 2 \omega_1 t  - i\epsilon}
\chi_{-}(\omega_1)  d \sigma_{\omega} 
= \tilde{c}_d\Re \int ^{2\delta}_{-2\delta} 
\frac{ \chi_{-}(-\sqrt{1-s^2})}{t^2 - 2t \sqrt{1-s^2}   - i\epsilon}
\frac{s^{d-2}}{\sqrt{1-s^2}} ds
\end{equation*}
for some constant $\tilde{c}_d$.
Here again we have integrated over the level surfaces of $\omega_1$,
a space of dimension $d-2$.  Applying Lemma \ref{l:gcont2} completes the 
proof. \end{proof}

\begin{lemma}\label{l:B2simp} Let $V\in L^\infty_c(\Real^d;\Real)$, and $B_2$ be as defined
in (\ref{eq:Bs}).  Then for $\lambda \in (0,\infty)$,
$$\frac{d}{d\lambda}
\Im \tr B_2(\lambda) = \frac{1}{2} (2\pi)^{-2d+1}\lambda^{2(d-2)} 
\frac{d}{d\lambda} \int 
|\hat{V}(\lambda \eta)|^2 g(|\eta|) d\eta + O(\lambda^{d-4})\; \text{
as $\lambda \rightarrow \infty.$}$$
Here $g(|\eta|)$ is the function defined in Lemma \ref{l:gprop}. 
\end{lemma}
\begin{proof}
Using (\ref{eq:B2expr})
\begin{align*}
\frac{d}{d\lambda} \Im \tr B_2 (\lambda)  & 
= 2(d-2) \lambda^{-1} \Im \tr B_2(\lambda) \\ & \hspace{2mm}
+\frac{\lambda^{2(d-2)}}{2 (2\pi)^{2d-1}}
 \frac{d}{d\lambda}\lim_{\epsilon\downarrow 0}
\Re \int_{\omega \in \Sphere^{d-1}}\int
|\hat{V}(\lambda \eta)|^2  \frac{1}{|\eta+ \omega|^2-1-i\epsilon} d\eta
d\sigma_\omega
\end{align*}
Using the estimate (\ref{eq:B2easyestimate}), the first term on the right above
is $O(\lambda^{d-4})$.  

Since $V\in L^\infty_0(\Real^d)$, then 
$\hat{V},\; \frac{\partial}{\partial \eta_j}\hat{V}(\eta)
 \in C^\infty(\Real^d) \cap L^2(\Real^d)$, and we can change
the order of the limit and the integral, getting, by Lemmas \ref{l:nodeltas} and
\ref{l:gprop} 
$$ \frac{d}{d\lambda}\lim_{\epsilon \downarrow 0}
\Re\int_{\omega \in \Sphere^{d-1}} \int
|\hat{V}(\lambda \eta)|^2  \frac{1}{|\eta+ \omega|^2-1-i\epsilon} d\eta
d\sigma_\omega
=\frac{d}{d\lambda}\int
|\hat{V}(\lambda \eta)|^2  g(|\eta|) d\eta
$$ where $g(|\eta|)$ is the function defined in Lemma \ref{l:gprop}.
\end{proof}

We may now prove Lemma \ref{l:B2est}.
\begin{proof}
Recall that $\hat{V}\in L^2(\Real^d)$ is smooth.  
Using Lemma \ref{l:B2simp} we must estimate
\begin{equation}\label{eq:B2deriv}
\frac{d}{d\lambda}\int
|\hat{V}(\lambda \eta)|^2 g(|\eta|) d\eta
= \sum_{j=1}^d
\int  \eta_j\left[ \hat{V}_j(\lambda \eta) \overline{\hat{V}}
(\lambda \eta) + \hat{V}(\lambda \eta) \overline{\hat{V}_j}
(\lambda \eta)\right] g(|\eta|) d\eta
\end{equation}
where $$\hat{V}_j(\eta) =\frac{\partial}{\partial \eta_j}\hat{V}(\eta).$$
We have used the properties of
 $\hat{V}$, $\frac{d}{d\eta_j}\hat{V}$, and $g$ to justify the 
  interchange the order of differentiation and integration.

We take just one term on the right in (\ref{eq:B2deriv}), as the 
others are handled in exactly the same way.
Introducing polar coordinates
\begin{equation}
\int  \eta_j  \hat{V}_j(\lambda \eta) \overline{\hat{V}}
(\lambda \eta) g(|\eta|) d\eta 
= \int_0^\infty \int_{\theta \in \Sphere^{d-1}} \rho \theta_j
\hat{V}_j(\lambda\rho \theta) \overline{\hat{V}}(\lambda\rho \theta)  g(\rho)
\rho^{d-1} d\sigma_\theta d\rho.
\end{equation}
We will write this integral as the sum of three 
integrals, depending on the size of $\rho$:
$0\leq \rho \leq 2-\lambda^{-1}$, $\rho \geq 2+\lambda^{-1}$,
and $|2-\rho|\leq \lambda^{-1}$.
Recalling that $\rho |\rho-2|^{1/2} g(\rho)$ is bounded when 
$\rho \leq 4$ 
we may bound
\begin{multline*} 
\left|\int_0^{2-\lambda^{-1}} \int_{\theta \in \Sphere^{d-1}}\theta_j \rho
\hat{V}_j(\lambda\rho \theta) \overline{\hat{V}}(\lambda\rho \theta) 
 g(|\rho|)
\rho^{d-1} d\sigma_\theta d\rho \right| \\
\leq C \lambda^{1/2}
\int_0^{2-\lambda^{-1}}\int_{\theta \in \Sphere^{d-1}}
 |\hat{V}_j(\lambda\rho \theta) \overline{\hat{V}}(\lambda\rho \theta) |
\rho^{d-1}d\sigma_\theta d\rho  .
\end{multline*}
Integrating over all of $\rho\in (0,\infty)$ and doing a change of variables
gives
\begin{align}\label{eq:B2,1}& 
 \left|\int_0^{2-\lambda^{-1}} \int_{\theta \in \Sphere^{d-1}}
\hat{V}_j(\lambda\rho \theta) \overline{\hat{V}}(\lambda\rho \theta) 
\theta_j \rho g(|\rho|)
\rho^{d-1} d\sigma_\theta d\rho \right| \nonumber \\ & 
\leq C \lambda^{1/2-d} \int_0^\infty  \int_{\theta \in \Sphere^{d-1}}
\left|\hat{V}_j(\rho \theta) \overline{\hat{V}}(\rho \theta) \right|
\rho^{d-1}  d\sigma_\theta d\rho \nonumber 
\\ & \leq C \lambda^{1/2-d} \| V\| \| x_j V\|.
\end{align}

Similarly, when $\rho >2$ we can, by Lemma \ref{l:gprop},
 bound $|g(\rho)|\leq C |\rho-2|^{-1/2}\rho^{-1}$.  Thus
\begin{align}\label{eq:B2,2}
& \left|\int_{2+\lambda^{-1}}^\infty \int_{\theta \in \Sphere^{d-1}}
\hat{V}_j(\lambda\rho \theta) \overline{\hat{V}}(\lambda\rho \theta) 
\theta_j \rho g(|\rho|)
\rho^{d-1} d\sigma_\theta d\rho \right| \nonumber 
\\ &  \leq C \lambda^{1/2} 
\int _{2+\lambda^{-1}}^\infty \int_{\theta \in \Sphere^{d-1}}
\left|\hat{V}_j(\lambda\rho \theta) \overline{\hat{V}}(\lambda\rho \theta)\right |
\rho^{d-1} d\sigma_\theta d\rho  \nonumber \\
& \leq C\lambda^{1/2} \int_0^\infty \int_{\theta \in \Sphere^{d-1}}
\left|\hat{V}_j(\lambda\rho \theta) \overline{\hat{V}}(\lambda\rho \theta) \right|
\rho^{d-1} d\sigma_\theta d\rho \nonumber \\
& \leq \lambda^{1/2-d} \int_{\Real^d} \left|\hat{V}_j(\eta) \overline{\hat{V}}
(\eta) \right| d\eta \nonumber \\
& \leq C \lambda^{1/2-d} \| V\| \| x_j V\|.
\end{align}

Now we consider the region with $2-\lambda^{-1}\leq \rho \leq 
2+\lambda^{-1}$.  Here we use  that for compactly supported $V$
$$ \int_{\theta \in \Sphere^{d-1}}
\left|\hat{V}(\lambda  \theta)\right|^2 dS_{\theta}
\leq C \lambda^{-(d-2)-1}\| V  \|^2$$
where the constant $C$ depends on the support of $V$; this is essentially
our bound on $\lambda^{-(d-2)/2}\Gamma_0(\lambda)\chi $ applied to $V$, see 
(\ref{eq:gammaest}).  
Hence 
\begin{align}\label{eq:B2,3}
& \left|\int_{2-\lambda^{-1}}^ {2+\lambda^{-1}}\int_{\theta \in \Sphere^{d-1}}
\hat{V}_j(\lambda\rho \theta) \overline{\hat{V}}(\lambda\rho \theta) 
\theta_j \rho g(|\rho|)
\rho^{d-1} d\sigma_\theta d\rho \right| \nonumber \\ & 
\leq C \int_{2-\lambda^{-1}}^ {2+\lambda^{-1}}
 \int_{\theta \in \Sphere^{d-1}}
\left| \hat{V}_j(\lambda\rho \theta) \overline{\hat{V}}(\lambda\rho \theta) 
\right| |\rho-2|^{-1/2}
\rho^{d-1} d\sigma_\theta d\rho \nonumber \\ & 
\leq C \|V\|_{L^1} \| x_j V\|_{L^1} \lambda^{-(d-2)-1} \int_{2-\lambda^{-1}}^ {2+\lambda^{-1}}
|\rho-2|^{-1/2}d\rho  \nonumber\\ & 
\leq C \|V\||_{L^1} \| x_j V\||_{L^1} \lambda^{-d+1-1/2} = O(\lambda^{-d+1/2}).
\end{align}
Using (\ref{eq:B2,1}), (\ref{eq:B2,2}),  and (\ref{eq:B2,3}) shows
that (\ref{eq:B2deriv}) is of order $O(\lambda^{-d+1/2})$ as 
$\lambda \rightarrow
\infty$.  Hence, by Lemma \ref{l:B2simp},  $\Im \tr B_2(\lambda) = O(\lambda^{d-7/2})$, finishing
the proof of the lemma.
\end{proof}

\section{The determinant of the scattering matrix near $0$}
\label{s:near0}
In the previous section we proved a result about the behavior of the logarithmic
derivative of the determinant of the scattering matrix when $\lambda\rightarrow
\infty$, $\lambda \in (0,\infty)$.  We next consider the behavior of the 
same quantity, but for $\lambda$ near $0$.
\begin{lemma} \label{l:near0} Let $d\geq 4$ be even.  Let 
$V\in L^\infty_c(\Real^d;\Real)$, with $S$  denoting the scattering
matrix  of $-\Delta +V$.  Let $\mathcal{P}_0$ denote
projection onto the $L^2$ null space of $-\Delta +V$, with ${\mathcal{P}}_0=0$ if 
the $L^2$ null space is empty.  If $d=4$, assume  Hypothesis \ref{hy:H1} holds.
Then
$$\frac{ \frac{d}{d\lambda} \det S(\lambda)}{\det S(\lambda)} = 
-\frac{\pi i (d-4)}{(2\pi)^{d}}\vol (\Sphere^{d-1}) \lambda^{d-5}\| \mathcal{P}_0V\|^2 + O(\lambda^{d-3}|\log \lambda|)
\; \text{as $\lambda \rightarrow 0$},\; \arg \lambda =0.$$
If $d=4$ but without assuming Hypothesis \ref{hy:H1}, 
\begin{equation}
\label{eq:noH1}
\frac{ \frac{d}{d\lambda} \det S(\lambda)}{\det S(\lambda)} = 
O(\lambda^{-1}|\log \lambda|^{-2}), \; \text{as $\lambda \rightarrow 0$},\; \arg \lambda =0.\end{equation}
\end{lemma}
Before proving the lemma, we make several comments.

We note that  (\ref{eq:noH1}) is proved in 
 \cite{SBeven}.  We  include an
outline of the proof of (\ref{eq:noH1}) here for 
completeness and for the convenience of the reader.

Note that the term $\langle \mathcal{P}_0 V, V \rangle= \| \mathcal{P}_0V\|^2 
$ may 
be nonzero (this nonzero contribution has gone
unnoticed in some places). Suppose $\phi \in L^2(\Real^d)$, 
$(-\Delta +V)\phi=0$.  Then for $R$ sufficiently large, 
$$\langle V,\phi \rangle = 
\int V(x) \overline{\phi}(x)dx
= \int _{|x|<R} \overline{V\phi}(x) dx 
= \int _{|x|<R} \overline{\Delta \phi}(x) dx 
= \int_{|x|=R} \frac{\partial}{\partial|x|} \overline{\phi}.$$
This is independent of $R$ (sufficiently large), so one may consider
evaluating it in the limit as $R\rightarrow \infty$.
Writing $r=|x|$, if $\frac{\partial}{\partial r} \phi = c_0 r^{1-d}+o(r^{1-d})$  as 
$r\rightarrow \infty$, then the integral is not $0$ if $c_0$ is 
not zero.  This may happen, for example, if $V$ is a radial potential
and $\phi$ is an element of the null space associated to the zero mode
on $\Sphere^{d-1}$.
On the other hand, if $\frac{d}{dr} \phi =o(r^{1-d})$ when $r \rightarrow
\infty$, then $\langle V,\phi \rangle =0$. In dimension $d=4$ an
element of the null space of $-\Delta +V$ which has nonzero component
in the $0$th spherical harmonic when expanded at infinity is not in $L^2$.
However, in dimension $d=4$ an element of the $L^2$ null space of
$-\Delta +V$ has radial derivative decaying at least as fast as 
$O(r^{-2-\sqrt{3}})$ at infinity.  Hence in dimension $d=4$, $\mathcal{P}_0V=0$.

\begin{proof} First we suppose that either $d\not = 4$ or, if $d=4$, 
Hypothesis \ref{hy:H1} holds.
By results of \cite{jensen>4,jensen4}, 
near $0$, with $0\leq \arg \lambda \leq \pi$, the resolvent $R_V(\lambda)$ 
satisfies
\begin{equation}\label{eq:RVat0}
R_V(\lambda) =-\lambda^{-2} \mathcal{P}_0+ \log \lambda\; B_{0,1}+ B_{0,0}+ O(|\lambda|^2 |\log \lambda|^2),\end{equation}
for some operators $B_{j,k}$ which are bounded
from $L^2_{c}(\Real^d)$ to $L^2_{\operatorname{loc}}(\Real^d)$.  Moreover,
the
operator $\frac{d}{d\lambda}R_V(\lambda)$ also has an expansion near $\lambda=0$, and the expansion can be found by formally 
differentiating (\ref{eq:RVat0}), 
with error term
 for the derivative
$O(1)$.  We note that by results of \cite{m-s},  for 
any $M>0$, $R_V(\lambda)$ 
has an expansion in powers of $\lambda$ and $\log \lambda$ that is 
valid for $|\arg \lambda|<M$.

We shall use the expression for the scattering matrix (\ref{eq:scattmatrix}).
Writing $\Gamma_0(\lambda, \omega,x)$ for the Schwartz kernel of 
$\Gamma_0(\lambda)$, near $\lambda=0$
\begin{equation}\label{eq:teg}
\Gamma_0(\lambda,\omega,x)V(x)= 2^{-1/2}(2\pi)^{-d/2}\lambda^{(d-2)/2}( 1-i\lambda x\cdot \omega+ O(|\lambda|^2))V(x)
\end{equation}
with $O(|\lambda|^2)$ error uniform in $x$ and $\omega$,
since $V$ has compact support.  Thus, using in addition that
$\Sphere^{d-1}$ is compact, the same
error holds for the corresponding operators using the $L^2\rightarrow L^2$
norm  or
the Hilbert-Schmidt norm. 
Then
\begin{equation}\label{eq:near01}
\left( \Gamma_0(\lambda) V R_V(\lambda)V\Gamma_0^*(\lambda)\right)(\omega,
\theta)=\frac{1}{2(2\pi)^{d}} \lambda^{d-2} \left[ 
-\frac{ \langle \mathcal{P}_0 V,V\rangle}{\lambda^2} + i\frac{ f(\omega)}{\lambda} -i\frac{ f(\theta)}{\lambda}
+ O(|\log \lambda|)\right]
\end{equation}
near $\lambda =0$.  Here
$$f(\theta) =\int V(x) x\cdot \theta (\mathcal{P}_0V)(x)  dx.$$
The error $O(|\log \lambda||\lambda|^{d-2})$ here 
holds for the corresponding operators in the $L^2 \rightarrow L^2$
 or the trace norm.
Moreover, 
formally differentiating (\ref{eq:near01})
 results in an expansion for the derivative of the 
left hand side near $0$, with the resulting error bounded
by $O(|\lambda|^{d-3})$.

We use 
\begin{align}
\frac{ \frac{d}{d\lambda}\det S(\lambda)}{\det S(\lambda)}=
\tr( S^*(\overline{\lambda})\dot{S}(\lambda)).
\end{align} 
Using the expression (\ref{eq:scattmatrix}) and the fact
that $\mathcal{P}_0 V=0$ if $d=4$,
we see that
$$\tr( S^*(\overline{\lambda})\dot{S}(\lambda))
= \tr \left[ 2\pi i \frac{d}{d\lambda} \left( \Gamma_0(\lambda) V R_V(\lambda)V
\Gamma_0^*(\lambda)\right) \right] +O(|\lambda|^{d-3})
$$
 near $\lambda =0$ with $\arg \lambda$ near $0$.
Using (\ref{eq:near01}) and the fact that here
the trace is given by the integral of the Schwartz kernel
over the diagonal $\omega=\theta$, this means
\begin{align*} & 
\tr( S^*(\overline{\lambda})\dot{S}(\lambda))\\ & 
= \frac{\pi i}{(2\pi)^d} \int_{\Sphere^{d-1}} \left[ 
-(d-4) \lambda^{d-5} \langle \mathcal{P}_0 V, V\rangle 
+i (d-3) \lambda^{d-4}(f(\theta)-f(\theta)) +O(|\lambda|^{d-3} |\log \lambda|)
\right] d\sigma_\theta\\ 
& = \frac{-\pi i}{(2\pi)^d} (d-4) \Vol(\Sphere^{d-1})  \lambda^{d-5}  
\|\mathcal{P}_0 V\|^2
+ O(|\lambda|^{d-3} |\log \lambda|).
\end{align*}

We outline how to modify the proof if $d=4$  without the 
assumption of Hypothesis \ref{hy:H1}.   In this case, by \cite{jensen4}
near $\lambda=0$, $0\leq \arg \lambda<\pi$,
$$R_V(\lambda) =-\lambda^{-2} \mathcal{P}_0 + 
\lambda^{-2} (a-2\log \lambda)^{-1}B_{-2,-1}+\log \lambda B_{0,1}+B_{0,0}+O(1/|\log \lambda|),$$
where $a$ is a constant depending on $V$ and the operators $B_{j,k}$ are bounded
from $L^2_{c}(\Real^d)$ to $L^2_{\operatorname{loc}}(\Real^d)$.  This expansion can also be 
differentiated, with resulting error $O(1/|\lambda \log \lambda|)$.
Using again that ${\mathcal P}_0V=0$ in dimension $d=4$, we get
$$\left( \Gamma_0(\lambda) V R_V(\lambda)V\Gamma_0^*(\lambda)\right)(\omega,
\theta)= \frac{b}{(a-2 \log \lambda)}+O(|\lambda|)$$
for some constant $b$.  Continuing as in the previous case, we prove
(\ref{eq:noH1}).
\end{proof}

\section{Writing $F$ in terms of canonical products}
\label{s:canonicalprod}
We turn now more directly to the proof of Theorem \ref{thm:samepoles}.  As
a next step, we write $F$, and then $F'/F$, using canonical products.
The proof uses many of the same components as the proofs of Propositions
2.1 and 2.2 of \cite{SBeven}.
In particular, the proof of the next proposition is very similar to that of
\cite[Proposition 2.1]{SBeven}, which itself uses 
techniques as in, for example, \cite{zworskieven}.  We use some of the  notation of \cite{SBeven}
to highlight the similarities.  

For $z\in \Complex$ set 
$$E_0(z)= (1-z)\;\text {and} \;
E_p(z)= (1-z) \exp\left(\sum _{k=1}^p \frac{z^k}{k}\right) \; \text{for $p\in \Natural$}.$$ 
Central to the proof 
is the observation
 that while the scattering matrix $S_j(\lambda)$ for $-\Delta +V_j$ is
a meromorphic function of $\lambda\in \Lambda$, $S_j(e^z)$ is a meromorphic
function of $z\in \Complex$.

For the next proposition, we use much of the notation of Theorem 
 \ref{thm:samepoles}.  Note, however, that we do not need assumption 
(\ref{eq:notmanypoles1}), nor do we need any 
additional hypotheses for the $d=4$ case.
\begin{prop}
\label{prop:gest}
 Let $d$ be even, and  $V_j\in L^{\infty}_c(\Real^d;\Real)$ for $j=1,\;2.$  Set
$P_j=-\Delta+V_j$ and let $S_j(\lambda)$ be the associated scattering 
matrix, unitary for $\lambda>0$. Set
$$F(z)=\frac{\det S_1(e^z)}{\det S_2(e^z)}.$$
Let $\{z_l\}$ denote the distinct poles of $F(z)$, and $M(z_l)$ their 
multiplicities. Suppose that there is an $m\in (0,\infty)$ such that 
\begin{equation}\label{eq:sumconv}
\sum \frac{M(z_l)}{|z_l|^m}<\infty.
\end{equation}
Let $m_1\in \Natural\cup \{ 0\}$ be the smallest nonnegative integer
satisfying $m_1+1\geq m$, and set 
$$P(z,m)= \prod\left( E_{m_1 }(z/\overline{z_l})\right)^{M(z_l)},\; 
Q(z,m)= \prod\left(E_{m_1}(z/z_l)\right)^{M(z_l)}.$$
Then 
$$F(z) = e^{g(z)} \frac{P(z,m)}{Q(z,m)}$$
where $g$ is an entire function satisfying $|g(z)|\leq C \exp(C|z|)$,
some $C>0$.
\end{prop}
\begin{proof} Since $\det S_j(e^z) \det S_j^*(e^{\overline{z}})=1$, the set 
 $\{ \overline{z_l}\}$ is the set of zeros of $F(z)$.
Hence $F(z)Q(z,m)/P(z,m)$ is an entire nowhere zero function,
so the only thing to prove is the 
bound on $|g(z)|$.

An intermediate result of the proof of \cite[Proposition 2.1]{SBeven} is that for every $R>1$ there is
a $\rho_j=\rho_j(R)\in (R/2,R)$ so that 
\begin{equation} \label{eq:sjbd}
|\det S_j(e^z)|\leq C\exp(\exp(C |z|)),\; \text{when $|z|=\rho_j$}
\end{equation}  
with constant $C$ independent of $R$.
For our application we will need to know that we can choose 
$\rho_1=\rho_2 \in (R/2,R)$ so that (\ref{eq:sjbd}) holds, and to 
understand this we explain the origin of the $\rho_j$.   The
need to choose $\rho_1=\rho_2$  is the 
main point of divergence from the proof of \cite[Propostion 2.1]{SBeven},
and we outline enough of the proof to show how to make the modifications 
necessary.

Let $\chi \in C_c^\infty(\Real^d)$ satisfy $\chi V_j =V_j$ for $j=1,2$.
Then we can write
$$S_j(\lambda)= I +A_j(\lambda)$$
where for a nonzero constant $c_d'$
$$A_j(\lambda) = c_d'E^\chi(e^{i\pi }\lambda) (I+V_j R_0(\lambda)\chi)^{-1} V_j
(E^\chi(\lambda))^t$$
and
$$ E^\chi(e^{i\pi }\lambda) =\Gamma_0(\lambda) \chi.
$$
For a bounded linear operator $B$, let $\mu_1(B)\geq \mu_2(B)\geq ...$
denote the characteristic values of $B$.
  Then
\begin{equation}\label{eq:charvalueform}
|\det S_j(\lambda)| \leq \prod_{l=1}^\infty (1+\mu_l(A_j(\lambda)))$$
and $$\mu_l(A_j(\lambda)) \leq |c_d'|
\mu_l (E^\chi(e^{i\pi \lambda}\lambda) )  \| (I+V_j R_0(\lambda)\chi)^{-1} \| \|V_j\|_{L^\infty}
\|((E^\chi(\lambda))^t\|.
\end{equation}
Only two terms involve $V_j$, and for $j=1,2$,
$\|V_j\|_{L^\infty}\leq C$ for some $C$.  Thus to prove (\ref{eq:sjbd})
we need to find regions where we can bound 
$\| (I+V_j R_0(\lambda)\chi)^{-1} \| $ independently of $j$.

As in \cite{zworskieven} (and just as in \cite{SBeven}), we use 
\cite[Theorem V.5.1]{g-k},
\begin{equation}\label{eq:gk}
\| (I+V_j R_0(e^z )\chi)^{-1} \|
\leq (1+ \| V_j R_0(e^z)\chi\|^{d/2})\frac{
\det (I + |V_j R_0(e^z)\chi|^{d/2+1})}{ \left| \det (I + (V_j R_0(e^z)\chi)^{d/2+1})\right| }.\end{equation}
We have, by choosing the constant $C$ to be 
the larger of the corresponding constants for $j=1,2$,
\begin{equation}\label{eq:ub1}
\left| \det (I + (V_j R_0(e^z)\chi)^{d/2+1})\right|\leq 
 \det (I + |V_j R_0(e^z)\chi|^{d/2+1}) \leq 
\exp(C\exp[(d+1)|z|])
\end{equation}
by \cite[Proposition 2.1]{intissar}.  

Let $f$ be an analytic function in the disc $\{z\in \Complex:\; |z|\leq 2eR\}$, 
with $f(0)=1$.  Then by Cartan's estimate (e.g. \cite[Theorem I.11]{levin}),
 if $0<\eta<3e/2$, then 
outside a family of discs the sum of whose radii does not exceed $4\eta R$,
$$\log |f(z)|>-(2+\log \frac{3e}{2\eta})\log{\mathfrak M}_f(2eR), \;
\text{where}\; {\mathfrak M}_f(s)=\max\{|f(z)|:\; |z|\leq s\}.$$

Now we apply Cartan's estimate to 
$f_j(z)=\det(I + (V_j R_0(e^z)\chi)^{d/2+1})$ , choosing $\eta =1/40$.
 This ensures that
we can find $\rho_1=\rho_2 \in (R/2,R)$  so that for a constant
$C$ independent of $R$, we have
\begin{equation}\label{eq:bothjs}
\left| \det (I + (V_j R_0(e^z)\chi)^{d/2+1})\right| \geq 
 \exp( -C\exp( C |z|)),\; \text{for}\;
j=1,2,\; \text{if}\;
|z|=\rho_1=\rho_2.\end{equation}
Otherwise, the sum of the sums of the radii of the exceptional
circles for
$f_1$ and for $f_2$ would exceed $(R-R/2)/2=R/4$.  But
if $\eta =1/40$, then  $2(4 \eta R)=
R/5$.

Now the techniques of \cite{zworskieven}, (\ref{eq:charvalueform}),
(\ref{eq:gk}), (\ref{eq:ub1}),  and (\ref{eq:bothjs}) show that 
(\ref{eq:sjbd}) holds for $\rho_1=\rho_2\in (R/2,R)$.
It follows from 
the identity $S_j(\lambda)S^*_j(\overline{\lambda})=I$ that 
$$\det(S_j(e^z))\overline{\det(S_j(e^{\overline{z}}))}=1.$$
Hence (\ref{eq:sjbd}) gives the same bound on the reciprocal of
$\det(S_j(e^z))$ on the circle $|z|=\rho_1(R)$, and 
\begin{equation}\label{eq:rhobd}
\left| \frac{\det S_1(e^z)}{\det S_2(e^z)}\right| \leq C \exp( \exp C|z|)),\; 
\text{if}\;|z|=\rho_1=\rho_2.
\end{equation}


The bounds on canonical products (e.g. \cite[Section I.4]{levin}) mean that 
for $\delta>0$
\begin{equation}\label{eq:cpb}
|Q(z,m)|\leq C \exp(C|z|^{m+\delta}),\; |P(z,m)|\leq C \exp(C|z|^{m+\delta})
\end{equation} with constants depending on $\delta$.
Applying this together with (\ref{eq:rhobd}) and using that we have chosen $\rho_1=\rho_2\in (R/2,R)$
gives
$$\left| F(z) Q(z,m)\right| \leq C \exp(\exp(C |z|))),\;\text{if}\; |z|=\rho_1.$$
But since $F(z) Q(z,m)$ is entire and we can find such a $\rho_1=\rho_2\in (R/2,R)$ for each $R>1$,
by the maximum principle we get $|F(z)Q(z,m)|\leq C\exp( \exp(C|z|))$.  Now Cartan's estimate
applied to the function $P(z,m)$ shows that
if $\delta>0$ then for each $R>1$ there is a $\rho'
=\rho'(R)\in (R/2,R)$ 
so that $|P(z,m)|\geq C\exp(- C' |z|^{m+\delta})$ when $|z|=\rho'$, 
with constants $C, \; C'$ independent of $R$. 
Thus 
$$|\exp(g(z))| \leq C \exp( C \exp |z|)),\; \text{if}\;|z|=\rho'.$$
Again using the maximum principle, the fact that   $g$ is entire,
and our ability to find such a $\rho'$ for each $R>10$, 
$$|\exp(g(z))|\leq C \exp (\exp C |z|)),\; z\in \Complex.$$
This implies that 
$$\Re g(z) \leq C \exp (C|z|).$$
By Carath\'eodory's theorem (e.g. \cite[Theorem I.8]{levin}), 
$$|g(z)| \leq C \exp (C|z|).$$
\end{proof}

Note that the hypotheses of Theorem \ref{thm:samepoles} ensure that $\sum_{M(z_k)>0}
|z_k|^{-(1-3\epsilon_0/4)}<\infty$.
\begin{lemma}\label{l:exptype1}
Under the hypotheses of Theorem \ref{thm:samepoles},
set $P_1(z)=\prod_{M(z_k)>0} E_0 (z/\overline{z_k})$ and $Q_1(z)= \prod_{M(z_k)>0} E_0(z/z_k)$.
 Then there is a constant $C$ so that
$$\left| \frac{F'(z)}{F(z) } P_1(z)Q_1(z) \right| 
\leq C \exp (C|z|).$$
\end{lemma}
\begin{proof} 
We use the functions $P$, $Q$, and $g$ from Proposition \ref{prop:gest},
noting our assumption on the poles of $F$ in Theorem  \ref{thm:samepoles}
includes that the sum (\ref{eq:sumconv}) converges for some finite value of $m$.  

By Proposition \ref{prop:gest},
\begin{equation}
\frac{F'(z)}{F(z)}= g'(z) +\frac{P'(z,m)}{P(z,m)} - \frac{Q'(z,m)}{Q(z,m)}.
\end{equation}
Hence
\begin{multline}\label{eq:wanttobd}
\frac{F'(z)}{F(z)}P_1(z)Q_1(z)= g'(z)P_1(z)Q_1(z) +
\frac{P'(z,m)}{P(z,m)} P_1(z)Q_1(z)- \frac{Q'(z,m)}{Q(z,m)}P_1(z)Q_1(z).
\end{multline}
Note that $P'/P$ and $Q'/Q$ have simple poles which coincide with the 
zeros of $P_1$ and $Q_1$, respectively.  Hence $F' P_1Q_1/F$ is entire.
Since $|g(z)|\leq C\exp(C|z|)$, by Cauchy's estimate the same inequality holds
for $g'(z)$, with perhaps a new constant $C$.  Moreover, $P_1(z)$, $Q_1(z)$
satisfy (\ref{eq:cpb}) with $m=1-\epsilon_0/2$.  Hence 
it is easy to see that the first term in (\ref{eq:wanttobd}) is bounded 
as claimed.

Consider
\begin{align*}\frac{P'(z,m)}{P(z,m)} P_1(z)
 & =-\left(  \sum_k 
\frac{M(z_k)}{\overline{z_k}}\frac{( z/\overline{z_k})^{m_1}}
{1-z/\overline{z_k} }\right)
\left( \prod_l ( 1-z/\overline{z_l}) \right)\\
& = -\sum_k
\left( \frac{M(z_k)z^{m_1}}{\overline{z_k}^{m_1+1}} \prod_{l:\;l\not = k}( 1-z/\overline{z_l})\right).
\end{align*}

Using the bounds on canonical products and the assumption on $\{z_l\}$ we
can bound 
$$\left|\prod_{l:l\not = k}( 1-z/\overline{z_l})\right|\leq C \exp (C |z|^{1-\epsilon_0/4})$$
with a constant $C$ chosen independent of $k$. 
 But
$$\sum_k \left| \left(  \frac{M(z_k)z^{m_1}}
{\overline{z_k}^{m_1+1} }
 \right)\right| 
=\left(  \sum_k \frac{M(z_k)}{|z_k|^{m_1 +1}} \right)
 |z|^{m_1}
\leq C |z|^{m_1}$$
for some constant $C$
using our assumptions on the convergence of 
$\sum M(z_k)/|z_k|^{ m }$.
Thus the second term is bounded as desired.  The third term in 
(\ref{eq:wanttobd}) is bounded in exactly the same way as the second.
\end{proof}

\section{Proof of Theorem \ref{thm:samepoles} }\label{s:samepoles}
The next proposition proves Theorem \ref{thm:samepoles} when
$d=2$, and is an important step in the proof of Theorem \ref{thm:samepoles}
for $d\geq 4$.  Here we use the notation $\delta_{j,k}$ for the Kronecker delta
function.  In the statement of the proposition, we understand $\lambda^{d-4},\;
\lambda^{d-2}\in \Complex$.  More carefully this might be denoted 
$(\pr(\lambda))^{d-4},\; (\pr(\lambda))^{d-2}$, where $\pr: \Lambda\rightarrow
\Complex$ is the natural projection, where points of $\Lambda$ with argument
differing by an integral multiple of $2\pi$ are identified.
\begin{prop}\label{p:almostsamepoles} Under the assumptions and using the notation of 
Theorem \ref{thm:samepoles}, $F(z)$ is an entire function.
Moreover, if $d\geq 4$
\begin{equation}\label{eq:almostequal}
\det S_1(\lambda)   = \exp\left(- i
 c_d (1-\delta_{d,4})\alpha_1 \lambda^{d-4} - i
c_d \alpha_2  \lambda^{d-2}\right)
\det S_2(\lambda)
\end{equation}
where $c_d= \pi(2\pi)^{-d} \vol(\Sphere^{d-1}) $.
Here
\begin{align*}
\alpha_1 & =  \| {\mathcal P}_{0,1} V_1\|^2 - \| {\mathcal P}_{0,2} V_2 \|^2\\
\alpha_2 & = \int(V_1(x)-V_2(x))dx,
\end{align*}
with $\mathcal{P}_{0,j}$
denoting projection onto the $L^2$ null space of $P_j$.  If $d=2$, 
$\det S_1(z)=\det S_2(z)$.
\end{prop}
\begin{proof}
We first prove the proposition assuming either that $d\not =4$ or $d=4$ and 
Hypothesis \ref{hy:H1} holds.

Consider the function defined by 
\begin{multline}\label{eq:G}
G(z)= 
e^{-(d-9/4)z} \left( \frac{F'(z)}{F(z)}+i(d-4)(1-\delta_{d,2})c_d\alpha_{1}e^{(d-4)z}
+i(d-2)c_d \alpha_2 e^{(d-2)z}
\right) \\ \times \prod_{M(z_l)>0} (1-z/z_l)(1-z/\overline{z_l}).
\end{multline}
To motivate our definition of $G$, notice that by Lemma 
\ref{l:near0} the first term after
$F'(z)/F(z)$, when evaluated at $z=x\in \Real$, is (up to sign)
 the leading term of
$F'(x)/F(x)$ when $x\rightarrow -\infty$.  By Theorem
\ref{thm:derivasmpt} the next term corresponds
to the leading term when $x\rightarrow \infty$.
Note that when $d=2$, the coefficients of both $\alpha_1$ and $\alpha_2$
are $0$, and when $d=4$ the coefficient of $\alpha_1$ is $0$.  The 
multiplication by $\prod(1-z/z_l)(1-z/\overline{z}_l)$ ensures the
function $G$ is analytic.

By our assumptions on $\{z_l\}$ and estimates for canonical products (e.g. 
\cite[Theorem 1.7]{levin}), 
there is a constant $C$ so that 
\begin{equation}\label{eq:prodbd}
\left| \prod_{M(z_l)>0} (1-z/z_l)(1-z/\overline{z_l}) \right| \leq C 
\exp (C|z|^{1-\epsilon_0/2}).
\end{equation}
Combining this with the result of Lemma \ref{l:exptype1}, we find
there is a constant $C$ so that 
\begin{equation}\label{eq:exptype}
|G(z)|\leq Ce^{C|z|}.
\end{equation}
Since  $G$ is analytic, $G$ is a function of exponential type (see e.g. \cite{boas}).

By Theorem \ref{thm:derivasmpt} and (\ref{eq:prodbd}), 
 for $x\in \Real$, $|G(x)|=O(e^{-x/4+C|x|^{1-\epsilon_0/2}})$ as $x\rightarrow
\infty.$  

Next we consider the behavior of $G(x)$ when $x\in \Real$, 
$x\rightarrow -\infty$.  If $d\geq 4$, by Lemma \ref{l:near0}
then $G(x)= O(|x| e^{x/4+C|x|^{1-\epsilon_0/2}})$ as $x\rightarrow -\infty$.  
If $d=2$, results of \cite{b-g-d} (when $\int V_j(x)\not =0$)
or \cite{chobstacle} (for $\int V_j=0$), imply that
$F'(x)/F(x)= O(x^{-2}) $ when $x\rightarrow -\infty$.  Combining these
asymptotics with 
(\ref{eq:prodbd}) implies
 that for $x\in \Real$, there is a constant 
$C>0$ such that $G(x)= O(\exp(x/4+ C|x|^{1-\epsilon_0/2}))$ when $x\rightarrow
-\infty$.

Now noting the bound (\ref{eq:exptype}) on $|G(z)|$ and the fact that $G(z)$
decays exponentially on the real axis, \cite[Corollary 5.1.14]{boas}
shows that $G\equiv 0$.  Hence $F'(z)/F(z)\equiv 0$ if
$d=2$, and 
$$F'(z)/F(z)= -i(d-4)c_d\alpha_{1}e^{(d-4)z}
- i(d-2)c_d \alpha_2 e^{(d-2)z} \; \text{if $d\geq 4$}.$$   Returning
to the $\lambda$ variable 
we find, for $d\geq 4$
$$\frac{\frac{d}{d\lambda} \det S_1(\lambda)}{
\det S_1(\lambda)} = 
\frac{\frac{d}{d\lambda} \det S_2(\lambda)}{
\det S_2(\lambda)} - i (d-4)c_d \alpha_1 \lambda^{d-5}- 
i c_d(d-2) \alpha_2 \lambda^{d-3}.$$
Recalling that
$\lim_{\lambda \in \Real_+, \lambda \downarrow 0}\det S_j(\lambda)=1$
(\cite{b-g-d} or \cite[Proposition 6.1]{chobstacle})
finishes the proof if $d\not =4$ or if $d=4$ and Hypothesis \ref{hy:H1} holds.

To complete the proof, we consider the remaining case, which is $d=4$ 
and $V_1,V_2\in C_c^\infty(\Real^4;\Real)$.  
In this case we define, in analogy with (\ref{eq:G}),
$$G_4(z) =e^{z} \left( \frac{F'(z)}{F(z)}+ i2c_4 \alpha_2 e^{2z}
\right)   \prod_{M(z_l)>0} (1-z/z_l)(1-z/\overline{z_l}).$$
Using that $V_1,\; V_2\in C^\infty_c(\Real^4;\Real)$, by \cite{gu,pop},
$$(\det S_{j}(\lambda))^{-1}\frac{d}{d\lambda}\det S_{j}(\lambda)
= -2ic_4\int V_j(x')dx' \lambda + O(\lambda^{-N}), \lambda >0,\; \lambda \rightarrow \infty$$
for any $N\in \Natural$.  Hence for $x>0$,  $G_4(x)=O(e^{(2-N)x})$ as $x\rightarrow
\infty$ for any $N\in \Natural$.  On the other hand, by Lemma 
\ref{l:near0}, if $x\in \Real$, $G_4(x)=O(x^{-2}e^{x+x-x})=O(x^{-2}e^{x})$ as 
$x\rightarrow -\infty$.  Now using that $G_4$ is an entire function of
exponential type decaying exponentially on the 
real axis, as before   \cite[Corollary 5.1.14]{boas}
shows that $G_4\equiv 0$.  The remainder of the proof follows as in the 
first case.
\end{proof}

The function $f$ defined below will appear in the proof of Lemma \ref{l:klimit}.
\begin{lemma}\label{l:flemma}
Set
$$f(t,r)=\frac{2t}{1-2rt+2r^2t}\; \text{for $t\in[0,2]$, $r\geq 1$.}$$
Then 
for fixed $r_0\geq 1$, $f(t,r_0)$ is an increasing function of $t\in[0,2]$,
and $0\leq f(t,r_0)\leq 4/(1-4r_0+4r_0^2)= 4/(1-2r_0)^2$.
For fixed $t_0\in [0,2]$, $f(t_0,r)$ is a decreasing function 
of $r\geq 1$, $lim_{r\rightarrow \infty}f(t_0,r)=0$ and $f(0,r)=0$.
\end{lemma}
\begin{proof}
These properties are immediate from inspection and elementary calculus.
\end{proof}

Although we state the following lemma for the determinant of the scattering
matrix for a Schr\"odinger operator, it is valid for a much larger class of 
operators.  In fact, if $P$ is an appropriate self-adjoint ``black-box'' perturbation
of $-\Delta$ on $\Real^d$, $d$ even, then the following lemma is valid for
the scattering matrix of $P$.  See \cite{sj-zw} for the definition of the 
black-box perturbation.  (For the following to hold, though, 
we need in addition the existence of an involution on the underlying 
Hilbert space which 
commutes with $P$ and which agrees with complex conjugation 
``at infinity.'') 
Note that in the following 
lemma, since $\rho>0$, $k\in \Natural$, the point $\rho e^{i\pi k}\in \Lambda$
projects in the complex plane to the real axis, so that the square of the 
projection lies in the continuous
spectrum of $P$.
\begin{lemma}\label{l:klimit}
Let $d$ be even, $V\in L^\infty_c(\Real^d;\Real)$ and let 
$S(\lambda)$ denote the scattering matrix 
associated to $-\Delta +V$.  Let $k\in \Natural$ and fix $\rho>0$.
Then $\lim _{ k\rightarrow \infty} \det S(\rho e^{i k\pi})=1$.
\end{lemma}
\begin{proof}
Since
$$\det S(e^{ik\pi}\lambda)=
\frac{ \det S(\lambda) \det S(\lambda e^{i\pi})\cdot \cdot \cdot \det S(\lambda e^{i(k-1)\pi})\det S(\lambda e^{ik\pi})}
{\det S(\lambda) \det S(\lambda e^{i\pi})\cdot \cdot \cdot 
\det S(\lambda e^{i(k-1)\pi})}
$$
we have using \cite[Lemma 3.4]{ch-hi4} 
(compare the proof of \cite[Proposition 3.5]{ch-hi4})
that
\begin{align*}
\det S(e^{ik\pi}\lambda)& = \frac{\det ( (k+1) S(\lambda)-kI)}
{\det ( k S(\lambda)-(k-1)I)}\\ & =\det \left( 
[ (k S(\lambda)-(k-1)I + S(\lambda)-I ] [k S(\lambda)-(k-1)I]^{-1} 
\right)
\\& 
 = 
\det(I + [S(\lambda)-I][I+k(S(\lambda)-I)]^{-1}).
\end{align*}
Now we specialize to $\lambda =\rho\in (0,\infty)$, and recall that
$S(\rho)$, being unitary, has eigenvalues $\{ e^{i\theta_j}\}$, 
$\theta_j \in \Real$.  
The $\theta_j$ depend on $\rho$, but since $\rho$ is fixed, 
we do not denote this dependence.  

Then
$$\det S(e^{ik\pi}\rho)= \prod_j 
\left(1+ \frac{e^{i\theta_j}-1}{1+k(e^{i\theta_j}-1)}
\right).$$
For $k\geq 1$
\begin{equation}
 | 1+k(e^{i\theta}-1)| \geq 1
\end{equation}
so that $|(e^{i\theta_j}-1)(1+k(e^{i\theta_j}-1))^{-1}|\leq |e^{i\theta_j}-1|$.
Hence given $\epsilon >0$, using the trace class properties of
$S(\tau)-I$  we can find $J\in \Natural$ independent of $k\in \Natural$
 so that
\begin{equation}\label{eq:smalltheta}
\left| 1-\prod_{j>J} \left(1+ \frac{e^{i\theta_j}-1}{1+k(e^{i\theta_j}-1)}
\right) \right| < \epsilon/5.
\end{equation}
By a straightforward computation
$$\left| \frac{e^{i\theta_j}-1}{1+k(e^{i\theta_j}-1) }\right| ^2
= f(1-\cos \theta_j, k)
$$ 
where $f$ is the function defined in Lemma \ref{l:flemma}.  
Then
$$\lim_{k\rightarrow \infty}
\left| \frac{e^{i\theta_j}-1}{1+k(e^{i\theta_j}-1)} \right| =0
\; \text{for all $j\leq J$}.
$$
Since the product over $j\leq J$ is a finite product, this means
$\lim_{k\rightarrow \infty}\prod_{j\leq J} 
\left(1+ \frac{e^{i\theta_j}-1}{1+k(e^{i\theta_j}-1)}\right)=1$, and 
we can find $K\in \Natural $ so that 
$$ \left| 1-\prod_{j\leq J} 
\left(1+ \frac{e^{i\theta_j}-1}{1+k(e^{i\theta_j}-1)}
\right) \right| < \epsilon/5\; \text{for $ k\geq K$,\; $k\in \Natural$}.$$
Together with (\ref{eq:smalltheta}) this shows that 
for $1>\epsilon >0$ 
$$\left | 1- \det S(e^{ik\pi}\rho)\right|
= \left|  1-\prod_{j} 
\left(1+ \frac{e^{i\theta_j}-1}{1+k(e^{i\theta_j}-1)}\right) \right| \leq \epsilon
\; \text{for}\; k\geq K,\; k\in \Natural.$$
Since $\epsilon>0$ was arbitrary, this concludes the proof.
\end{proof}

There does not seem to be anything analogous to Lemma \ref{l:klimit} in odd
dimensions.  This lemma allows us to improve
Proposition \ref{p:almostsamepoles} to  Theorem \ref{thm:samepoles}.

\vspace{2mm}
\noindent{\em Proof of Theorem \ref{thm:samepoles}.} 
Proposition \ref{p:almostsamepoles} has proved most of Theorem 
\ref{thm:samepoles}, including the
$d=2$ case.  Thus, suppose $d\geq 4$.  From Proposition \ref{p:almostsamepoles},
$$\det S_1(\lambda) =e^{iq(\lambda)} \det S_2(\lambda) $$
where $q:\Lambda\rightarrow \Complex$ is given by 
\begin{equation}
q(\lambda)=-
 c_d (1-\delta_{d,4})\alpha_1 \lambda^{d-4} - 
c_d \alpha_2  \lambda^{d-2} .
\end{equation}
It remains to show that $e^{iq(\lambda)}\equiv 1$.

Notice that for $\rho>0$, $k\in \Natural$, $q(\rho e^{i\pi k})=q(\rho)$ since $d-2$, $d-4$ are even.  Hence, for $\rho>0$, $k\in \Natural$,
$\det S_1(\rho e^{i\pi k})= e^{i q(\rho)} \det S_2(\rho e^{i \pi k}).$
Fixing $\rho>0$, by Lemma \ref{l:klimit} $\lim_{k\rightarrow \infty }
\det S_1(\rho e^{i\pi k})=1=\lim_{k\rightarrow \infty }
\det S_2(\rho e^{i\pi k})$, so that $e^{i q(\rho)}=1$.  Since this is true 
for all $\rho>0$, we must have $e^{iq}\equiv 1$.
\qed

\section{Proofs of Theorems \ref{thm:Hm}, \ref{thm:infinitelymany}, and 
 \ref{thm:heatcoefficients} and Corollary \ref{cor:infinitediff} }
\label{s:thmpfs}  With Theorem \ref{thm:samepoles} in hand, 
the remainder of the proof of Theorem \ref{thm:Hm} is similar to the proof of 
\cite[Theorem 1.2]{sm-zw} and the proof of Theorem 
  \ref{thm:infinitelymany}
is similar to the proof of  \cite[Theorem 1.1]{SBeven}.  
 We include the proofs for the convenience of the 
reader.

\vspace{2mm}

\noindent {\em Proof of Theorem \ref{thm:Hm}.}
Recall that $P_j=-\Delta +V_j$ for $j=1,\; 2$, and set $P_0=-\Delta$.
Let $-\mu_{l,j}^2,$ $l=1,...,L_j$ denote the non-positive eigenvalues
of $P_j$, repeated according to multiplicity.  For $t>0$, we recall 
(\ref{eq:b-k}),
\begin{equation}\label{eq:bkf}
\tr(e^{-tP_j}-e^{-tP_0})=
\frac{1}{2\pi i} \int_0^\infty
\tr\left( S_j^{-1}(\lambda)\frac{d}{d\lambda} S_j(\lambda)\right)e^{-t\lambda^2} d\lambda
+ \sum_{l=1}^{L_j} e^{t\mu_{l,j}^2} +\beta(V_j,d).
\end{equation}


Using that 
$\det S_1(\lambda)=\det S_2(\lambda)$ by Theorem \ref{thm:samepoles},
and (\ref{eq:bkf})
\begin{equation}\label{eq:tr1}
\tr( e^{-tP_1}-e^{-tP_2}) =
 \sum_{l=1}^{L_1} e^{t\mu_{l,1}^2}-\sum_{l=1}^{L_2} e^{t\mu_{l,2}^2} 
+ \beta(V_1,d)-\beta(V_2,d).
\end{equation}
In particular,  by (\ref{eq:tr1}) 
\begin{equation}
\label{eq:tr2} \tr( e^{-tP_1}-e^{-tP_2})= t^{-(d-2)/2}f(t)\; \text{
with $f(t)\in C^{\infty}([0,\infty))$.}
\end{equation}
Note that  the parity of $d$ is important here.

By \cite[Theorem 4]{sm-zw} and our assumption that $V_2\in H^k$, there
are constants $c_1,c_2,...,c_{m+1}$ so that
$$\tr(e^{-tP_2}-e^{-tP_0})= (4 \pi t)^{-d/2}(c_1t+c_2t^2+...+c_{k+1}t^{k+1} +r_{k+2}(t)t^{k+2})\; \text{when $t\downarrow 0$}$$
with $|r_{k+2}(t)|\leq C$ for $0\leq t \leq 1$. 
But combining this with  (\ref{eq:tr2}) 
  implies that that there are constants $\tilde{c}_1,\tilde{c}_2,...,\tilde{c}_{k+1}$ and 
a function $\tilde{r}_{k+2}(t)$ so that
$$\tr(e^{-tP_1}-e^{-tP_0})= (4 \pi t)^{-d/2}(\tilde{c}_1t+\tilde{c}_2t^2+...+
\tilde{c}_{k+1}t^{k+1} + 
\tilde{r}_{k+2}(t)t^{k+2})\; \text{when $t\downarrow 0$}$$
with $|\tilde{r}_{l+2}(t)|\leq C$ for $0\leq t \leq 1$. 
Hence, again by \cite[Theorem 4]{sm-zw}, $V_1\in H^k(\Real^d)$.
\qed

\vspace{2mm}

A natural question to ask is the following: suppose $P_1$ and $P_2$ satisfy
the conditions of Theorem \ref{thm:samepoles}.
Theorem \ref{thm:samepoles} implies that  $\det S_1$ and
$\det S_2$ have exactly the same zeros and poles on $\Lambda$, 
including multiplicity.
It is natural, then, to ask if   $P_1$ and $P_2$ have the same 
resolvent resonances away from $0$.  
While it seems likely that they do, it is possible
to describe a scenario in which the symmetric difference of their 
(resolvent) resonance sets, where elements are repeated with 
multiplicity, is infinite.  This sort of scenario is the setting
of Lemma \ref{l:multiplicities}, which is used in the proof of Theorem
\ref{thm:infinitelymany}.

\vspace{2mm}
\noindent
{\em Proof of Theorem \ref{thm:infinitelymany}.}  The proof is 
by contradiction. Suppose that for some nontrivial potential $V$ and some $\epsilon >0$
we have $$\lim \sup_{r\rightarrow \infty}
\frac{N(r)}{(\log r)^{1-\epsilon}}<\infty.$$   If $d=4$, suppose
in addition that Hypothesis \ref{hy:H1} holds.
Let $S$ denote the scattering matrix of $-\Delta +V$.  Since the poles of 
$\det S(\lambda) $ are a subset of the poles of the resolvent, 
by   Theorem 
\ref{thm:samepoles} with $V=V_1$ and
$V_2\equiv 0$ $\det S(\lambda)$ has no poles.   Then we can apply 
 Theorem \ref{thm:Hm} with $V_1=V$ and $V_2\equiv 0$ to see that
 $V\in H^m(\Real^d)$ for all $m\in \Natural$.  Hence $V\in C_c^{\infty}(\Real^d;\Real)$.

Now we essentially follow \cite{SBeven} to show that there must be infinitely
many poles of the cut-off resolvent of $-\Delta+V$.    We fill in a few details omitted
in \cite{SBeven}.
By our Theorem \ref{thm:samepoles}
 (see also the proof of \cite[Theorem 1.1]{SBeven}),
$$\det S(\lambda) = 1.
$$ 
As in 
the proof of Theorem \ref{thm:Hm} we consider the heat trace 
$H(t)$ for $t>0$.   Using $\det S(\lambda)\equiv 1$ in 
(\ref{eq:b-k}), we find
\begin{equation}
\label{eq:heatcons}
H(t) =\tr \left( e^{-t(-\Delta +V)}- e^{t\Delta} \right) 
= 
 \sum _{k=1}^Ke^{t\mu_k^2} +\beta(V,d),\; t>0
\end{equation}
where $-\mu_1^2,...,-\mu_K^2$ are the non-positive eigenvalues of $-\Delta +V$.

It is  well-known that
\begin{equation}\label{eq:heatwellknown}
H(t) \sim \sum _{j=1}^\infty C_j(V) t^{j-d/2},\; t\downarrow 0.
\end{equation}
For us it is important to note that 
\begin{equation}\label{eq:Cs}
C_2(V) = \alpha_{2,d}\int V^2(x)dx
\end{equation} with nonzero 
constant $\alpha_{2,d}$; see e.g. \cite{b-sb,CdV} and references therein.
If $d>4$, using that $C_2(V)\not =0$ 
 we have immediately a contradiction between
(\ref{eq:heatcons}) and (\ref{eq:heatwellknown}), showing 
that $\det S(\lambda)$ must have, by Theorem \ref{thm:samepoles},
infinitely many poles.  In fact, we 
must have 
$$\lim\sup N(r)/(\log r)^{1-\epsilon}= \infty \; 
\text{ for all $\epsilon >0$,}$$
since otherwise the hypotheses of Theorem \ref{thm:samepoles} hold.  Since the 
poles of $\det S(\lambda)$ are a (perhaps proper) subset of the 
poles of the resolvent by \cite{ch-hi4}, we prove the theorem when $d>4$.
Note that in this case we have actually proved a potentially stronger result than
claimed in the statement of the theorem, as we have proved a lower bound on the 
counting function for the poles of the determinant of the scattering matrix.

Next we shall show that if $d=2$ or $d=4$ and Hypothesis \ref{hy:H1} holds,
 $\det S$ has no poles, and $V\not \equiv 0$, then $-\Delta +V$ must have 
at least one strictly negative eigenvalue.  For $d=2$, this follows
immediately again by comparing the expansions (\ref{eq:heatcons}) 
and (\ref{eq:heatwellknown}), and noting that $C_2(V)$, in the $d=2$
case the coefficient of $t$ in the 
expansion of $H(t)$ at $t=0$, is nonzero.  Hence there must be at least one 
negative eigenvalue.  

Now we turn to the case of $d=4$.  In \cite{b-y}, Benguria and Yarur showed that in dimension $d=3$, if $W\in L^\infty(\Real^3;\Real)$ goes to $0$
sufficiently rapidly
at infinity, then $0$ cannot be 
an eigenvalue of $-\Delta +W$ if $-\Delta +W$ does not have
a negative eigenvalue.  With a modification to their proof, in particular, 
a change to \cite[Theorem 2]{b-y} using the function $g(r) =r^{-2}$, one
can show the same is true in dimension $d=4$,
at least if the potential $W$ has
compact support.  Using this 
and returning to the case at hand, if $V\in C_c^\infty(\Real^4;\Real)$,
$V\not \equiv 0$, has $\det S$
analytic, since $C_2(V)\not =0$ by comparing (\ref{eq:heatcons})
and (\ref{eq:heatwellknown}) (recalling $\beta(V,4)=0$ by assumption)
 $-\Delta +V$ must have at least
one eigenvalue, and hence at least one negative eigenvalue.

Now in the $d=2$ and $d=4$ cases, we have shown that $-\Delta +V$ must
have at least one negative eigenvalue if $\det S(\lambda)$
has no poles and $V\not \equiv 0$.  Thus  by Corollary
\ref{c:negativeeigenvalues}  the cut-off resolvent of $-\Delta +V$
has infinitely many poles.  Moreover, the explicit 
location of the poles
shows that $\lim \sup_{r\rightarrow \infty}
\frac{N(r)}{(\log r)^{1-\epsilon}}=\infty$.
\qed

\vspace{2mm}

\noindent {\em Proof of Theorem \ref{thm:heatcoefficients}.}
Under the hypotheses of Theorem \ref{thm:heatcoefficients},
(\ref{eq:tr1}) holds.   However, we have also assumed that the negative
eigenvalues of $P_1$ and $P_2$ agree, as do their multiplicities.  Hence
for $t>0$
$$\tr(e^{-tP_1}-e^{-tP_2}) = {\mathfrak n}_0(P_1)+\beta(V_1,d)-
\mathfrak{n}_0(P_2)-\beta(V_2,d)$$
where ${\mathfrak n}_0(P_j)$ is the dimension of the $L^2$ null space of $P_j$.
However, we have defined $\mathfrak{n}_0(P_j)+\beta(V_j,d)$ to be the 
multiplicity of $0$ as a resonance of $P_j$, so that
$\tr(e^{-tP_1}-e^{-tP_2})=0$.  Therefore
\begin{align*}
0= \tr(e^{-tP_1}-e^{-tP_2}) 
\sim \sum_{l=1}^\infty (C_l(V_1)-C_l(V_2)) t^{l-d/2},\; \text {as $t\downarrow 0.$}
\end{align*}
This proves the result immediately.
\qed
\vspace{3mm}

To prove Theorem \ref{thm:inverse}, we shall use an intermediary step of 
considering Schr\"odinger operators on a flat torus $M$.  Given 
$R_0>0$, we shall define a corresponding flat torus $M=M(R_0)$ by identifying
opposite sides of $\{ x\in \Real^d: \max|x_j|\leq R_0+1\}$. Henceforth
we omit the $R_0$ dependence of $M$, as we shall hold $R_0$ fixed. If 
$V\in C^\infty_c(\Real^d;\Real)$ has its support in $B(0,R_0)$, then
we can consider
$V$ as an element of $  C^\infty_c(M;\Real)$ in a natural way. Thinking of 
$V\in C^\infty(M;\Real)$ gives a 
corresponding Schr\"odinger operator $P_{V,M}=-\Delta_M+V$ acting on 
(a domain in) $L^2(M)$, where $\Delta_M\leq 0$ is the Laplacian on the flat 
torus $M$.  Then it is well known that 
\begin{equation}
\tr_M e^{-tP_{V,M}}\sim t^{-d/2}\sum_{l=0}^\infty C_{l,M}(V)t^{-l} \; 
\text{as $t\downarrow 0$}
\end{equation}
where  $\tr_M$ denotes the trace on  $L^2(M)$.

Recall we  denote the heat coefficients of $V$ on $\Real^d$ by 
$C_l(V)$; see (\ref{eq:heatwellknown}).
\begin{lemma}\label{l:heatcoeff}
Suppose $V_1,\; V_2\in C_c^\infty(\Real^d;\Real)$ are supported in $B(0,R_0)$.
Then $C_l(V_1)=C_l(V_2)$ for all $l\in \Natural $ if and only if
$C_{l,M}(V_1)=C_{l,M}(V_2)$ for all $l\in \Natural$.
\end{lemma}
\begin{proof}
Let us denote by $P_1=-\Delta+V_1$ and $P_2=-\Delta+V_2$ the Schr\"odinger 
operators on $\Real^d$, and by $P_{1,M}=-\Delta_M+V_1$ and $P_{2,M}=-\Delta_M+V_2$ the Schr\"odinger operators on $M$.

Note that $C_l(V_1)=C_l(V_2)$ for all $l\in \Natural $ if and only 
if $\tr(e^{-tP_1}-e^{-tP_2})=O(t^N)$ for all $N$ as $t\downarrow \infty$.  Let
$\chi \in C_c^\infty(\Real^d)$ have its support in $B(0,R_0+1/2)$ and be equal to 
$1$ on $B(0,R_0+1/4)$. Then one can see, for 
example from the expressions for
the heat trace of \cite[Section 3]{sm-zw},  that
$$\tr(e^{-tP_1}-e^{-tP_2})- \tr(\chi e^{-tP_1}\chi -\chi e^{-tP_2}\chi)=O(t^N)
\; \text{as $t\downarrow 0$}$$
 for all $N\in \Natural$.
Note that 
$$\tr(\chi e^{-tP_1}\chi -\chi e^{-tP_2}\chi)
=\tr_M(\chi e^{-tP_1}\chi -\chi e^{-tP_2}\chi)$$
since $\chi$ is supported in $B(0, R_0 +1/2)$, and $M$ is
equipped with the flat metric.
But then using \cite[Lemma 1.5]{smi}, we see that
$$ \tr_M(\chi e^{-tP_1}\chi -\chi e^{-tP_2}\chi) 
- \tr_M(\chi e^{-tP_{1,M}}\chi -\chi e^{-tP_{2,M}}\chi) =O(t^N)\;\text{as 
$t\downarrow 0$}$$
for all $N\in \Natural$, and by a second application of \cite[Lemma 1.5]{smi}
 that
$$\tr_M(\chi e^{-tP_{1,M}}\chi -\chi e^{-tP_{2,M}}\chi)
- \tr_M( e^{-tP_{1,M}} - e^{-tP_{2,M}}) =O(t^N)\; \text{as $t\downarrow 0$}$$
for all $N\in \Natural$.  But 
$\tr_M( e^{-tP_{1,M}} - e^{-tP_{2,M}}) =O(t^N)\; \text{as $t\downarrow 0$}$ for all $N\in \Natural$
if and only if $C_{l,M}(V_1)=C_{l,M}(V_2)$ for all $l\in \Natural$.
\end{proof}

We shall need another lemma for our proof of Theorem \ref{thm:inverse}.
\begin{lemma}\label{l:isresaway0} For $n\in \Natural$, let 
$V_n\in L^\infty_c(\Real^d;\Real) $, with 
$\supp V_n\subset B(R_0,0)$, and suppose $\mu_{R_{V_n}}(\lambda_0)= 
\mu_{R_{V_m}}(\lambda_0)$ for all $\lambda_0\in \Lambda$ and all $n,\;m\in 
\Natural$.  Then if $V_n\rightarrow V_*$ in $L^\infty(\Real)$, then
$\mu_{R_{V_n}}(\lambda_0)= \mu_{R_{V_*}}(\lambda_0)$ 
for all $\lambda_0\in \Lambda$.
\end{lemma}
\begin{proof}
This proof is the same as in the odd-dimensional case.  We recall a proof as in \cite{hi-wo} 
for the convenience of the reader.

Since $V_n\rightarrow V_*$, $\supp V\subset \overline{B}(R_0,0)$. 
Let $\chi \in C_c^\infty(\Real^d)$ 
be $1$ on $\overline{B}(R_0,0)$.
For $\lambda_0\in \Lambda$, it is well known (cf. \cite[Section 3.4]{dy-zw}) that 
\begin{equation}\label{eq:multeq}
\mu_{R_{V_*}}(\lambda_0)=\msca(\det_p(I+V_*R_0(\lambda)\chi); \lambda_0),
\end{equation} where $p\in \Natural$,
$p>d/2$ is fixed and $\det_p$ is the  regularized determinant defined for 
operators of the type $I+B$, with $B$ in the $p$-Schatten class.  

Fix $p>d/2$, $p\in \Natural$, and set $h_n(\lambda)=  \det_p(I+V_nR_0(\lambda)\chi)$ and $h_*(\lambda)
= \det_p(I+V_*R_0(\lambda)\chi)$.  Then $h_n\rightarrow h_*$, uniformly on compact sets of
$\Lambda$.  Now we locally identify a neighborhood of $\lambda_0\in \Lambda$ with an open set in the
complex plane.
Given $\lambda_0$ in $\Lambda$, choose a small circle $\gamma_{\lambda_0}$ in $\Lambda$ so 
that no zeros of $h_*$ or of $h_n$ lie on $\gamma_{\lambda_0}$. We may in 
addition ensure that there are no zeros of  $h_n$ inside $\gamma_{\lambda_0}$, 
except, possibly, at $\lambda_0$.   This is possible since 
the zeros of both $h_*$ and $h_n$ are isolated, and the zeros of $h_n$ are independent of $n$.
Using Hurwitz's Theorem there is an $N\in \Natural$ so that if $n>N$, $h_n$ and $h_*$ have
the same number of zeros, counted with multiplicity, inside $\gamma_{\lambda_0}$.  By choosing $\gamma_{\lambda_0}$ appropriately, this
shows that if $\msca(h_n;\lambda_0)=0,$ then $\msca(h_*;\lambda_0)=0.$
Applying this again for other values of $\lambda_0$ shows that in general
$\msca(h_n;\lambda_0)=\msca(h_*;\lambda_0)$.
\end{proof}

\noindent {\em Proof of Theorem \ref{thm:inverse}.}
By Theorem \ref{thm:heatcoefficients}, if $V\in Iso(V_0,R_0)$, then $C_l(V)=C_l(V_0)$ for all 
$l\in \Natural$.  By Lemma \ref{l:heatcoeff}, $C_{l,M}(V)=C_{l,M}(V_0)$ for
all $l\in \Natural$.

We use results of Br\"uning \cite{bruning}
and Donnelly \cite{donnelly}.  Although the results of \cite{bruning,donnelly}
are stated as results for isospectral Schr\"odinger operators on compact 
manifolds, a careful reading shows that the proofs of the results
therein use the isospectrality of the Schr\"odinger operators only to 
show  
 that the  operators have the same heat coefficients,
and not any additional 
properties of isospectral Schr\"odinger operators. 
Suppose $\{V_n\}\subset Iso(V_0,R_0)$ if $d=2$ (respectively,
$\{V_n\}\subset Iso(V_0,R_0,s,c_0)$ if $d\geq 4$).
Hence, from  the compactness results 
of \cite{bruning, donnelly}, since the $V_n$ have the same 
heat coefficients on $M$,  $\{ V_n\}$ has a subsequence
which converges in $C^\infty(M)$ to a function $V_*$, necessarily in 
$C^\infty(M)$.  By selecting a subsequence and 
relabeling if necessary, we can assume $V_n\rightarrow V_*.$
Because $V_n$ has its support in $B(R_0,0)$ for each
$n$, so does $V_*$.
Since the heat coefficients $C_l(V)$ are 
continuous functions of $V$ and $C_l(V_n)=C_l(V_0)$ for each $l,\;n\in \Natural$, 
$C_l(V_*)=C_l(V_0)$ for each $l\in \Natural$.
 If $d\geq 4$,  $\|V_*\|_{H^s}\leq c_0$,
since this holds for each $V_n$.  It remains only to show that 
$-\Delta+V_*$ is isoresonant with $-\Delta +V_0$. 

By Lemma \ref{l:isresaway0},  $V_*$ and $V_0$ have the same nonzero resonances with 
the same multiplicities.  Recall that this means that they have the same negative eigenvalues
(if any) with the same multiplicities.  Using (\ref{eq:b-k}) and Theorem
\ref{thm:samepoles},  for $t>0$
$$\tr(e^{t(\Delta -V_*)}-e^{t(\Delta-V_0)})= 
\fn_0(-\Delta+V_*)+\beta(V_*,d)-\fn_0(-\Delta+V_0)-
\beta(V_0,d)$$
where $\fn_0(-\Delta+V_*)$ is the dimension of the $L^2$
null space of $-\Delta +V_*$, and similarly for $V_*$ replaced by $V_0$.
However, since $-\Delta +V_0$ and $-\Delta+V_*$ have the same heat coefficients,
for any $N\in \Natural $ there
is a $C>0$ so that $|\tr(e^{t(\Delta -V_*}-e^{t(\Delta-V_0)})|\leq C t^N$ for $0<t<1$.
Hence $$\fn_0(-\Delta+V_*)+\beta(V_*,d)=\fn_0(-\Delta+V_0)+
\beta(V_0,d).$$
But we have chosen the left (respectively right) hand side
 to be the definition of the multiplicity of $0$ as 
a resonance of $-\Delta +V_*$ (resp. $-\Delta +V_0$), completing the proof.
\qed

\vspace{2mm}


\noindent {\em Proof of Corollary \ref{cor:infinitediff}.}  
Suppose to the the contrary that there are potentials $V_1$, 
$V_2\in L^\infty_c(\Real^d;\Real)$ so that the nonzero poles of the 
meromorphically continued resolvents $R_{V_j}(\lambda)$, including 
multiplicities, differ by a nonzero finite number of elements.  Then
since the poles of determinants of the associated scattering matrices 
$S_j(\lambda)$ are a subset of the poles of $R_{V_j}$, the hypotheses 
of Theorem \ref{thm:samepoles} are fulfilled.  Hence by
Theorem \ref{thm:samepoles}, $\det S_1(\lambda)/\det S_2(\lambda)\equiv1$.

Now we use the notation of Lemma \ref{l:multiplicities}.
Suppose for some $\tau_0>0$, $\theta_0\in \Real$, 
$$\mu_{R_1}(\tau_0 e^{i\theta_0})-\mu_{R_2}(\tau_0 e^{i\theta_0})=m_0\not =0.$$
This must hold for some $\tau_0>0$, $\theta_0\in \Real$ if the set of 
(nonzero) poles of the resolvents of $P_1$ and $P_2$ are 
not identical.
But then by Lemma \ref{l:multiplicities}, 
$$\mu_{R_1}(\tau_0 e^{i(\theta_0+k\pi)})-\mu_{R_2}(\tau_0 e^{i(\theta_0+k\pi)})=m_0\not =0\; \text{for all $k\in \Integers$}.$$
But this contradicts the assumption that the difference of
the number of poles of the 
resolvents, counted with multiplicity, is a finite number.
\qed

\section{The singular part of $R_V(\lambda)$ at $\lambda_0 e^{i \pi k}$
and linear independence}
\label{s:independence} 
In this section we denote by $\pr$ the natural projection,
$\pr: \Lambda \rightarrow \Complex$, which identifies points whose arguments
differ by an integral multiple of $2\pi$.

Suppose $R_V(\lambda)$ has poles at $\lambda _1, \lambda_2\in \Lambda$,
with, for $j=1,2$ nontrivial $f_j$ in the range of the singular part of $R_V(\lambda)$ 
at $\lambda_j$, with $(-\Delta +V-(\pr(\lambda_j))^2)f_j=0$.
It is easy to see from the condition that $f_j$ is in the null space
of $-\Delta +V-(\pr(\lambda_j))^2$ that if $(\pr(\lambda_1))^2\not = 
(\pr(\lambda_2))^2 $
and $f_1$ and $f_2$ are both 
nontrivial, then they are linearly independent.  This argument
can easily be extended to $m$ distinct points $\{\lambda_1,...,\lambda_m\}$
in $\Lambda$,
as long as the set $\{ (\pr(\lambda_1))^2,..., (\pr(\lambda_m))^2\}$
consists of $m$ {\em distinct} points in the complex plane.
This argument
does not work, however, if some elements of the 
set $\{ (\pr(\lambda_1))^2,..., (\pr(\lambda_m))^2\}$
coincide.  Since
on $\Lambda$ there are infinitely many points which project to any  
 given point in the complex plane, the question of linear independence of the
ranges of the singular parts of the resolvent may be complicated.  
We do not attempt to fully answer this question here.  However, we 
do show that in some sense, made precise in Proposition 
\ref{p:eigenvalue} below, the ranges of the singular part of the resolvent for a set of 
distinct points in $\Lambda$ are ``usually'' linearly independent.

To state our results, we introduce a little notation. Suppose
$\lambda_1\in \Lambda$ is a pole of $R_V(\lambda)$ of order $l$.
Then we will say $f$ is in the range of the most singular part of 
$R_V(\lambda)$ at $\lambda_1$ if $f$ is in the range of 
$((\lambda-\lambda_1)^l R_V(\lambda))_{\restrict \lambda=\lambda_1}$.
Here we locally identify a neighborhood of $\lambda_1$ in $\Lambda$ with 
an open set in the complex plane so that $(\lambda -\lambda_1)^l$ makes
sense.


In even dimension $d$, for any $l\in \Integers$ 
\begin{equation}R_0(e^{i\pi l} \lambda)= R_0(\lambda)-il T(\lambda)
\end{equation}
where  $$\left( T(\lambda)f \right) (x) =\frac{1}{2} (2\pi)^{1-d}\lambda^{d-2}
\int \int e^{i\lambda \omega \cdot (x-y)} f(y) dy d\sigma_\omega$$
where $d\sigma_\omega$ is the usual measure on $\Sphere^{d-1}$.

\begin{lemma} \label{l:nullspace} Let $d$ be even, 
$V\in L^\infty_c(\Real^d;\Real)$, and $m\in \Integers, $ $m\geq 2$.
Suppose for some $\lambda_0=\rho e^{i\varphi}$, $0\leq \varphi <\pi$, $\rho>0$,
there are {\em distinct} integers $k_1,\; k_2,..., k_m$ and not identically $0$ 
 functions $f_{k_l}$
in the range of the most singular part of $R_V$ at $\lambda_0 e^{i\pi k_l}$,
so that
$\{ f_{k_1},\; f_{k_2},...,f_{k_l}\}$ are linearly dependent.
Then for each integer $p$, $R_V(\lambda)$ has a pole at $
\lambda=e^{ip \pi}\lambda_0$.
\end{lemma}
\begin{proof}  Since $f_{k_l}$ is in the image of the 
most singular part of $R_V$ at $e^{i\pi k_l}\lambda_0$, we have
$(-\Delta +V-\rho^2 e^{2i\varphi})f_{k_l}=0$.
Set $\phi_{k_l}=(-\Delta -\rho^2 e^{2i\varphi})f_{k_l}= -Vf_{k_l}.$  Since 
$\{f_{k_l}\}$ forms a linearly dependent set, so do $\{ \phi_{k_l}\}$.
Moreover, by unique continuation, none of the functions
$\phi_{k_l}$ are the zero function.
Since $R_V(\lambda)= R_0(\lambda) (I+VR_0(\lambda))^{-1}$
so that  $f_{k_l}$ is in the image of 
$R_0(\lambda_0 e^{i\pi k_l})$, the function
$\phi_{k_l}$ is in the null space of $I+VR_0(\lambda_0 e^{i\pi k_l})
= I+V(R_0(\lambda_0)-ik_l T(\lambda_0))$.

By relabeling and decreasing $m$ if necessary, we
can assume that no proper subset of the $\{ \phi_{k_l}\}$ is linearly dependent.
Then there are nonzero constants $c_2,...,c_m$ so that 
$$\phi_{k_1}=\sum_{l=2}^m c_{l} \phi_{k_l}.$$
But 
\begin{align*}
0&= (I+VR_0(\lambda_0)-ik_1 V T(\lambda_0))\phi_{k_1} \\
&=  (I+VR_0(\lambda_0)-ik_1 V T(\lambda_0))\sum_{l=2}^m c_{l} \phi_{k_l} \\
&= -i VT(\lambda_0)\sum_{l=2}^m c_{l} (k_1-k_l)\phi_{k_l};
\end{align*}
that is, $\sum _{l=2}^m c_{l} (k_1-k_l)\phi_{k_l}$ is in the null space of 
$VT(\lambda_0)$.
If $p\in\{k_1,...,k_m\}$ there is nothing to prove.  So we 
assume $p \not \in\{k_1,...,k_m\}$.  
Then
\begin{align*}
 (I+VR_0(\lambda_0)-ip VT(\lambda_0))
 \sum_{l=2}^m c_{l} \frac{(k_1-k_l)}{k_l-p}\phi_{k_l}
& = i\sum_{l=2}^m c_{l} (k_1-k_l) VT(\lambda_0) \phi_{k_l}
\\ & = iVT(\lambda_0)\sum_{l=2}^m c_{l} (k_1-k_l)\phi_{k_l}
\\ & =0.
\end{align*}
By our assumption that no proper subset of $\{ \phi_{k_1},...,\phi_{k_m}\}$
is linearly dependent,  the  function 
$g= \sum_{l=2}^m c_{l} \frac{(k_1-k_l)}{k_l-p}\phi_{k_l}$ is nontrivial.

Hence $I+VR_0(e^{i p \pi}\lambda_0)$ has a nontrivial null space.
But it is well known that $R_V(\lambda)$ has a pole whenever
$I+VR_0(\lambda)$ has nontrivial null space.
\end{proof}

\begin{prop}\label{p:eigenvalue}
If the hypotheses of Lemma \ref{l:nullspace}  hold, then
$\varphi=\pi/2$ and $-\rho^2$ is an eigenvalue of $-\Delta+V$. 
\end{prop}
\begin{proof}
Suppose the hypotheses of Lemma \ref{l:nullspace} hold.  Then by Lemma 
\ref{l:nullspace} $R_V(\lambda)$ has a pole at
$\lambda_0$, and $\lambda_0$ lies in the closure of the  physical region,
more particularly in the region with $0\leq \arg \lambda <\pi$.
But the only poles of $R_V(\lambda)$ with $0\leq\arg \lambda<\pi$ correspond
to the square roots of eigenvalues.  Since $V$ is real-valued, the eigenvalues of
$-\Delta +V$ are real, and $\varphi = \pi/2$.
\end{proof}



\vspace{2mm}
\noindent 
{\bf Acknowledgments.}  It is a pleasure to thank 
Rafael Benguria, Peter Hislop, 
Ant\^onio S\'a Barreto, and Maciej Zworski for helpful conversations.
The author is grateful to the Simons Foundation for its support through the Collaboration Grants for Mathematicians program.


\begin{thebibliography}{99}

\bibitem{autin} A. Autin, {\em Isoresonant complex-valued potentials and symmetries.}
Canad. J. Math. {\bf 63} (2011), no. 4, 721-754. 

\bibitem{b-sb} R. Ba\~nuelos and A. S\'a Barreto, {\em On the heat trace of 
Schr\"odinger operators.} Comm. Partial Differential Equations {\bf 20}
 (1995), no. 11-12, 2153-2164.

\bibitem{b-y}
R. Benguria and C. Yarur, {\em Sharp condition on the decay of the potential for the absence of a zero-energy ground state of the Schr\"odinger equation.}
J. Phys. A {\bf 23} (1990), no. 9, 1513-1518. 
\bibitem{boas} R.P. Boas, Jr, {\em Entire Functions}.   Academic Press,
New York, 1954.

\bibitem{b-d-o} D. Boll\'e, C. Danneels, and T.A. Osburn,
{\em Local and global spectral shift functions in $\Real^2$}.
J. Math. Phys. {\bf 30} (1989), no. 2, 420-432. 

\bibitem{b-g-d} D. Boll\'e, F. Gesztesy, and C. Danneels,
{\em Threshold scattering in two dimensions.} 
Ann. Inst. H. Poincar\'e Phys. Th\'eor. {\bf 48} (1988), no. 2, 
175-204.

\bibitem{bruning} J. Br\"uning, {\em On the compactness of isospectral potentials.}
Comm. Partial Differential Equations {\bf 9} (1984), no. 7, 687-698. 

\bibitem{chen} L.-H. Chen, {\em A sub-logarithmic lower bound for resonance counting function in two-dimensional potential scattering.} 
Rep. Math. Phys. {\bf 65} (2010), no. 2, 157-164. 
\bibitem{chex} T. Christiansen, 
{\em Schr\"odinger operators with complex-valued potentials and no resonances.
} Duke Math. J. {\bf 133} (2006), no. 2, 313-323.

\bibitem{iso} T. Christiansen, 
{\em Isophasal, isopolar, and isospectral Schr\"odinger operators and elementary
 complex analysis.} Amer. J. Math. {\bf 130} (2008), no. 1, 49-58.

\bibitem{chobstacle} T.J. Christiansen,
 {\em Lower bounds for resonance counting functions for 
obstacle scattering in even dimensions.}  Amer, J. Math.,
{\bf 139} (2017), no. 3, 617-640.

\bibitem{ch-hi2} T.J. Christiansen and P.D. Hislop, 
{\em Maximal order of growth for the resonance counting functions for generic potentials in even dimensions.}
 Indiana Univ. Math. J. {\bf 59} (2010), no. 2, 621-660.


\bibitem{ch-hi4} T.J. Christiansen and P.D. Hislop, 
 {\em Some remarks on resonances in even-dimensional Euclidean
scattering}.  Trans. Amer. Math. Soc. {\bf 368} (2016), no. 2, 1361-1385. 


\bibitem{CdV} Y. Colin de Verdi\`ere, {\em Une formule de traces pour l'op\'erateur de Schr\"odinger dans $\Real^3$}.
Ann. Sci. \'Ecole Norm. Sup. (4) {\bf 14} (1981), no. 1, 27-39. 

\bibitem{donnelly} H. Donnelly, {\em Compactness of isospectral potentials.} 
Trans. Amer. Math. Soc. {\bf 357} (2005), no. 5, 1717-1730. 

\bibitem{dy-zw} S. Dyatlov and M. Zworski, {\em Mathematical Theory
of Scattering Resonances.}  \url{http://math.mit.edu/~dyatlov/res/res_20170323.pdf}

\bibitem{froese} R. Froese, {\em Asymptotic distribution of resonances
in one dimension}, J.
Differential Equations {\bf 137} (1997), no. 2, 251-272.

\bibitem{froeseodd} R. Froese, {\em Upper bounds for the resonance
counting function of Schr\"odinger operators in odd dimensions}, 
Canad. J. Math.  {\bf 50} (1998),  no. 3, 538-546.

\bibitem{g-k} I.C. Gohberg and M.G. Krein, {\em Introduction to the theory of linear nonselfadjoint operators.} Translations of Mathematical Monographs, Vol. 18 American Mathematical Society, Providence, R.I. 1969.

\bibitem{gu} L. Guillop\'e, {\em
Asymptotique de la phase de diffusion pour l'op\'erateur de Schr\"odinger 
dans $\Real^n$}. 
 Bony-Sj\"ostrand-Meyer seminar, 1984-1985, Exp. No. 5, 11 pp., \'Ecole Polytech., Palaiseau, 1985.

\bibitem{hi-wo} P.D. Hislop and R. Wolf, {\em Compactness of iso-resonant potentials for Schr\"odinger operators in low odd dimension,} preprint.

\bibitem{ho1} L. H\"ormander, 
The analysis of linear partial differential operators. I. Distribution theory 
and Fourier analysis. Second edition. Grundlehren der Mathematischen 
Wissenschaften [Fundamental Principles of Mathematical Sciences], 256. 
Springer-Verlag, Berlin, 1990. 

\bibitem{intissar} A. Intissar, {\em A polynomial bound on the number of the scattering poles for a potential in even-dimensional spaces $\Real^n$.}
 Comm. Partial Differential Equations {\bf 11} (1986), no. 4, 367-396. 

\bibitem{jensen>4} A. Jensen, {\em Spectral properties of Schr\"odinger 
operators and time-decay of the wave functions. Results in $L^2(\Real^m),\; m
\geq 5$.}  Duke Math. J. {\bf 47} (1980), no. 1, 57-80.

\bibitem{jensen4} A. Jensen, {\em Spectral properties of Schr\"odinger operators and time-decay of the wave functions. Results in $L^2(\Real^4)$}.
J. Math. Anal. Appl. {\bf 101} (1984), no. 2, 397-422. 

\bibitem{jensenasymptotics} A. Jensen, {\em High energy asymptotics for the total scattering phase in potential scattering theory.}
 Functional-analytic methods for partial differential equations 
(Tokyo, 1989), 187-195,
Lecture Notes in Math., {\bf 1450}, Springer, Berlin, 1990. 


\bibitem{k-k} A. Komech and E. Kopylova,  
Dispersion decay and scattering theory. 
John Wiley \& Sons, Inc., Hoboken, NJ, 2012. 

\bibitem{khalfline} E. Korotyaev, {\em Inverse resonance 
scattering on the half line.} Asymptot. Anal. {\bf 37} (2004), no. 3-4, 215-226.

\bibitem{kstability} E. Korotyaev, {\em Stability for inverse
 resonance problem.} Int. Math. Res. Not. 2004, no. 73, 3927-3936. 
\bibitem{kline} E. Korotyaev, {\em Inverse resonance scattering on the 
real line.} Inverse Problems {\bf 21} (2005), no. 1, 325-341.

\bibitem{levin} B. Ja. Levin, {\em Distribution of zeros of entire functions},
American Mathematical Society, Providence, R.I. 1964.

\bibitem{m-s} J. M\"uller and A. Strohmaier, 
{\em The theory of Hahn-meromorphic functions, a holomorphic Fredholm 
theorem, and its applications.} Anal. PDE {\bf 7} (2014), no. 3, 745-770.

\bibitem{pop} G.S. Popov, {\em Asymptotic behaviour of the scattering phase for the Schr\"odinger operator.}
C. R. Acad. Bulgare Sci. {\bf 35} (1982), no. 7, 885-888. 

\bibitem{SBeven} A. S\'a Barreto, {\em Lower bounds for the number of resonances in even-dimensional potential scattering.}
 J. Funct. Anal. {\bf 169} (1999), no. 1, 314-323.

\bibitem{SaB01} A. S\'a Barreto,
{\em Remarks on the distribution of resonances in odd dimensional 
Euclidean scattering.} Asymptot. Anal. {\bf 27} (2001), no. 2, 161-170.

\bibitem{simon} B. Simon, {\em 
Resonances in one dimension and Fredholm determinants},
  J. Funct. Anal.  {\bf 178}  (2000),  no. 2, 396-420. 

\bibitem{sj-zw} J. Sj\"ostrand and M. Zworski, 
{\em Complex scaling and the distribution of scattering poles.}
J. Amer. Math. Soc. {\bf 4} (1991), no. 4, 729-769. 

\bibitem{smi} H. Smith, {\em On the trace of Schr\"odinger heat kernels and 
regularity of potentials}, preprint available at 
\url{https://sites.math.washington.edu/~hart/research/heat_trace_rm.pdf}

\bibitem{sm-zw} H. Smith and M. Zworski, {\em Heat traces and existence of scattering resonances for bounded potentials.}  Ann. Inst. Fourier (Grenoble)
{\bf  66} (2016), no. 2, 455-475.

\bibitem{vodeveven} G. Vodev, 
{\em Sharp bounds on the number of scattering poles in even-dimensional spaces.}
Duke Math. J. {\bf 74} (1994), no. 1, 1-17. 

\bibitem{vodev2} G. Vodev, 
{\em Sharp bounds on the number of scattering poles in the two-dimensional case.}
Math. Nachr. {\bf 170} (1994), 287-297. 

\bibitem{yafaevbook} D. Yafaev, 
Scattering theory: some old and new problems.
 Lecture Notes in Mathematics, 1735. Springer-Verlag, Berlin, 2000. 

\bibitem{yafaev} D. Yafaev, {\em The Schr\"odinger operator: perturbation determinants, the spectral shift function, trace identities, and more.}
  Funktsional. Anal. i Prilozhen. {\bf 41} (2007), no. 3, 60-83, 96; 
translation in
Funct. Anal. Appl. {\bf 41} (2007), no. 3, 217-236.

\bibitem{yafaevmathematical} D. Yafaev,
{\em Mathematical scattering theory.}
Analytic theory. Mathematical Surveys and Monographs, 158. American Mathematical Society, Providence, RI, 2010.

\bibitem{zworski1d} M. Zworski, {\em
Distribution of poles for scattering on the real line. }
J. Funct. Anal. {\bf 73} (1987), no. 2, 277-296. 

\bibitem{zworskieven} M. Zworski, {\em Poisson formula for resonances 
in even dimensions.}
Asian J. Math. {\bf 2} (1998), no. 3, 609-617. 

\bibitem{zwremark}M. Zworski, {\em A remark on isopolar potentials.}
SIAM J. Math. Anal. {\bf 32} (2001), no. 6, 1324-1326. 
\end{thebibliography}
\end{document}